\documentclass[12pt,reqno]{amsart}  

\usepackage{float}

\usepackage{amssymb,amsbsy,amsfonts,amsthm,latexsym,amsopn,amstext,amsxtra,euscript,amscd,amsthm,url}

\usepackage{epsfig, subfigure, afterpage}

\usepackage[usenames,dvipsnames,svgnames,table,x11names]{xcolor} 

\usepackage{hyperref}
\usepackage{cleveref} 

\usepackage{enumitem}

\usepackage{microtype} 
\usepackage[vmargin=2.2cm,hmargin=2.6cm]{geometry} 

\newtheoremstyle{plaintheorems}
{10pt}
{6pt}
{}
{}
{\bfseries}
{.}
{.5em}
{\thmname{#1}\thmnumber{ #2}\thmnote{ (#3)}}

\theoremstyle{plaintheorems}
\newtheorem{Rem}{Remark}

\newtheoremstyle{sltheorems}
{10pt}
{6pt}
{\slshape}
{}
{\bfseries}
{.}
{.5em}
{\thmname{#1}\thmnumber{ #2}\thmnote{ (#3)}}

\theoremstyle{sltheorems} 
\newtheorem{Thm}{Theorem}
\newtheorem{CThm}{Classical Theorem}
\newtheorem*{Thm*}{Theorem}

\newtheorem{conj}{Conjecture}
\newtheorem{Defi}{Definition}

\newtheorem{lem}{Lemma}

\newtheorem*{cor*}{Corollary}

\newcommand{\C}{\mathbb{C}}
\newcommand{\N}{\mathbb{N}}

\newcommand{\Z}{\mathbb{Z}}

\newcommand{\cal}{\mathcal}

\newcommand*\kronecker[2]{%
\relax\if@display
\expandafter{(\frac{#1}{#2})}
\else
\expandafter{(#1/#2)}
\fi
}
\makeatother

\makeatletter
\let\@@pmod\pmod
\DeclareRobustCommand{\pmod}{\@ifstar\@pmods\@@pmod}
\def\@pmods#1{\mkern4mu({\operator@font mod}\mkern 6mu#1)}
\makeatother

\newcommand{\be}{\begin{equation}}
\newcommand{\ee}{\end{equation}}

\newcommand{\eal}{\ensuremath{\end{align*}}}
\newcommand{\bea}{\begin{eqnarray}}
\newcommand{\eea}{\end{eqnarray}}

\newcommand{\eps}{\ensuremath{\varepsilon}}

\renewcommand{\AA}{\mathcal{A}}
\renewcommand{\le}{\leqslant}
\renewcommand{\ge}{\geqslant}
\renewcommand{\leq}{\leqslant}
\renewcommand{\geq}{\geqslant}

\DeclareMathAlphabet{\curly}{U}{rsfs}{m}{n}

\allowdisplaybreaks

\usepackage{mathtools} 
\newcommand{\e}{\mathbf{e}}
\newcommand{\Euler}{{\pmb \gamma}}

\usepackage{microtype} 
\usepackage{geometry}
\usepackage{caption} 
\makeatletter
\patchcmd{\env@cases}{1.2}{1}{}{}
\makeatother

\begin{document}
\title[Kummer ratio of the relative class number for cyclotomic fields]
{The Kummer ratio of the relative class number\\ for prime cyclotomic fields}

\author[N. Kandhil, A. Languasco, P. Moree, S. Saad Eddin and A. Sedunova]
{Neelam Kandhil, Alessandro Languasco, Pieter Moree, \\ Sumaia Saad Eddin 
and Alisa Sedunova}

\date{}

\begin{abstract}
\noindent 
Kummer's conjecture predicts the asymptotic growth of the relative class number of 
prime cyclotomic fields. 
We substantially improve the known bounds of 
Kummer's ratio under three scenarios: no Siegel zero, presence of Siegel 
zero and assuming the Riemann Hypothesis for the Dirichlet $L$-series attached to 
odd characters only.
The numerical work in this paper extends and improves on our earlier 
preprint \url{https://arxiv.org/abs/1908.01152} 
and demonstrates our theoretical results.
\end{abstract}

\maketitle 
\setcounter{tocdepth}{1}
\mbox{}\vspace{-1cm}
\section{Introduction}
Let $K$ be a number field, ${\cal O}$ its ring of integers and $s$ a complex variable.
For $\Re(s)>1$ the 
\emph{Dedekind zeta function} is defined by
\begin{equation*}
\zeta_K(s)=\sum_{\mathfrak{a}} \frac{1}{N{\mathfrak{a}}^{s}}
=\prod_{\mathfrak{p}}\frac{1}{1-N{\mathfrak{p}}^{-s}},
\end{equation*}
where $\mathfrak{a}$ ranges over 
the non-zero ideals in ${\cal O}$, 
$\mathfrak{p}$ ranges over the prime ideals in ${\cal O}$, and $N{\mathfrak{a}}$ denotes the 
\emph{absolute norm}
of $\mathfrak{a}$,
that is the index of $\mathfrak{a}$ in $\cal O$.
It is known that $\zeta_K(s)$ can be analytically continued to $\C \setminus \{1\}$,
and that it has a simple pole at $s=1$. Notice that $\zeta_{\mathbb Q}(s)$ equals $\zeta(s)$, the \emph{Riemann zeta function}.

Let  $q\ge 3$ be a prime number and $K=\mathbb Q(\zeta_q)$ 
a \emph{prime cyclotomic field}. Denote as
${\mathbb Q}(\zeta_q)^+:={\mathbb Q}(\zeta_q+\zeta_q^{-1})$  the {\it maximal real cyclotomic field}.
We have the factorization
\begin{equation}
\label{oud}
\zeta_{\mathbb Q(\zeta_q)}(s)=\zeta(s)\prod_{\chi\ne \chi_0}L(s,\chi),
\end{equation}
where $\chi$ runs over the non-principal characters modulo $q$.
Likewise we have 
\begin{equation}
\label{maxreal}
\zeta_{{\mathbb Q}(\zeta_q)^+}(s)=\zeta(s)\prod_{\substack{\chi \neq \chi_0 \\ \chi(-1)=1}} L(s,\chi),
\end{equation}
where the product is over all even characters modulo $q$.

\subsection{Kummer's conjecture} 
Let $h_1(q)$ be the ratio of the class number $h(q)$ of 
$\mathbb Q(\zeta_q)$ and the class number 
of its maximal real subfield ${\mathbb Q}(\zeta_q)^+$. 
Kummer proved that this is an integer. 
It is now called either the
\emph{relative class number}, or the \emph{first factor of the class number}, 
and played an important role in Kummer's research on Fermat's Last Theorem. Indeed, it is not difficult to 
show that if $q\nmid h(q)$, then $x^q+y^q=z^q$ has no non-trivial solution with $q$ coprime to $xyz$ \cite[Ch.\,1]{Wbook}. 
Kummer showed that $q$ divides $h(q)$ if and only if $q$ divides $h_1(q)$. As there is no easy way to compute $h(q)$, 
this is a very important result.
\par Some authors related $h_1(q)$ to a determinant and tried to estimate it in this way (cf.\,\cite[\S 2]{FungGW1992}).
Most famous is here the connection with the Maillet determinant
due to Carlitz and Olsen \cite{CO} (independently obtained by Chowla and Weil, 
who however did not publish their result). 
For any integer $n$ co-prime to $q$, let $n'$ be the smallest positive integer such that $nn' \equiv 1\pmod*{q}$ and let $A(n,q)$ 
be the smallest positive residue of $n$ modulo $q$. Let $M_q=(A(mn',q))_{1\leq m,n \leq (q-1)/2}$.
Then $\det(M_q)$ is called Maillet’s determinant. In 1955, Carlitz and Olson proved that $ \det(M_q) = \pm q^{(q-3)/2} h_1(q)$.
From this Carlitz \cite{Carlitz61} deduced the  appealing 
bounds $h_1(q) \le (\frac{q-5}{4})!$ when $q \equiv 1 \pmod*{4}$, and 
$h_1(q) \le (\frac{q-7}{4})! (\frac{q-3}{4})^{\frac{1}{2}}$ when $q \equiv 3 \pmod*{4}$. 
Much more recently Guo \cite{Guo} proved that
suitable normalizations ofthe determinants of $(\cot(jk\pi/q))_{j,k}$ 
and $ (\tan(jk\pi/q))_{j,k}$ for $1\le j,k\le (q-1)/2$ have $h_1(q)$ as factor.

\par In 1972, Mets\"ankyl\"a  \cite{Mets72} 
(simpler proof in \cite{Mets74}), established the elegant
bound  $h_1(q)<2q(q/24)^{(q-1)/4}$. 
In 1982, Feng \cite{Feng} showed using a determinantal approach that
\begin{equation*}
    h_1(q)<2q\Big(\frac{q-1}{31.997158}\Big)^{(q-1)/4}.
\end{equation*}
Fung et al.\,\cite{FungGW1992} used 
determinants to exactly compute $h_1(q)$ for $q<3000$, extending earlier computations by others. 
Kummer himself impressively computed by hand up to $q=163$, only making three mistakes. 

\begin{Defi}
Let $q$ be a prime number,
\begin{equation}
\label{Rq-def}
G(q):=2q\Bigl(\frac{q}{ 4\pi^2}\Bigr)^\frac{q-1}{4}, \quad
R(q):=
\frac{h_1(q)}{G(q)} \quad and \quad r(q):=\log R(q).
\end{equation}
\end{Defi}
The ratio $R(q)$ is called \emph{Kummer ratio}.
In 1851, Kummer \cite{Kummer1851} conjectured that $h_1(q)$ 
asymptotically grows in the same way as 
the elementary function $G(q)$.
\begin{conj}
\label{Kconjecture}
As $q$ tends to infinity,
$R(q)$ tends to $1$.
\end{conj}
For a generic prime $q$ the
Kummer ratio $R(q)$ is close to $1$ (see 
Sections \ref{sec:Kummer} and \ref{sec:numerical} for a theoretical, respectively numerical, underpinning).
However, in this paper our focus is on the extremal behavior of $R(q)$.
Our starting point in studying
$R(q)$ will be the identity
\begin{equation}
\label{hasse}
R(q)=\prod_{\chi(-1)=-1}L(1,\chi),
\end{equation}
where the product is over all the odd characters modulo $q$ (cf.\,Hasse \cite{Hasse}).
It follows from this, \eqref{oud} and \eqref{maxreal} that
\[
R(q)=\lim_{s \downarrow 1}\frac{\zeta_{\mathbb Q(\zeta_q)}(s)}
{\zeta_{{\mathbb Q}(\zeta_q)^+}(s)} = \lim_{s \downarrow 1} 
\prod_{\chi(-1)= -1}
L(s,\chi),
\] 
where $s \downarrow 1$ means that $s>1$ tends to $1$.
The reason why only the odd Dirichlet characters 
are involved into \eqref{hasse} follows from the fact
that, using \eqref{oud}-\eqref{maxreal}, in the ratio
$\frac{\zeta_{\mathbb Q(\zeta_q)}(s)}
{\zeta_{{\mathbb Q}(\zeta_q)^+}(s)}$
the Riemann zeta and the even characters contributions cancel out. We refer the interested reader to \cite{Full-K-ratio} to explore the bounds of the product $\prod_{\chi \ne \chi_0} L(1,\chi)$, where $\chi$ varies over all the non-principal characters modulo $q$.

Masley and Montgomery \cite[Thm.\,1]{MasMon} 
gave an \emph{effective} bound for $R(q)$, which in 
combination with numerical work, allowed them to  
prove Kummer's conjecture that $h(q)=1$ if
and only if $q\le 19$.
Using their method 
ineffective, but rather sharper, estimates for $R(q)$ were obtained by 
Puchta \cite{Puchta}\footnote{In Theorem 1 of \cite{Puchta} one should
read $(p+3)/4$ instead of $(p-3)/4$.} and more recently by 
Lu-Zhang  \cite{LuZhang} and Debaene \cite{deb}.

Our main result, Theorem \ref{rq-direct}, 
improves on all of these (see also Section \ref{sec:thm1-comments}
for more details).
It involves the \emph{exponential
integral function} (defined as $E_{1}(x):= \int_{x}^{\infty} e^{-t}\,dt/t$ for $x>0$), 
and the \emph{Siegel zero}
(defined in Section \ref{sec:Siegel-zero}).
\begin{Thm}
\label{rq-direct}
Let $\ell(q)$ be a function that tends arbitrarily slow and monotonically 
to infinity as $q$ tends to infinity.
There is an effectively computable prime
$q_0$ (possibly depending on $\ell$) and an effectively 
computable prime $q_1$  such that 
the following statements are true:
\begin{enumerate}[label={\arabic*)}, wide, nosep, itemindent= 4pt, after=\vspace{-12pt}]
\item
\label{nosiegelzero}
If for some $q\ge q_0$ the family 
of Dirichlet $L$-series $L(s, \chi)$, with $\chi$ any odd character modulo $q$, has no Siegel zero, 
for example if $q \equiv 1 \pmod*{4}$,
then
$$
\max\{R(q),R(q)^{-1}\}  < 
e^{0.41}
\, (\log q) \,\ell(q).
$$
\item
\label{Siegelzero} 
If for some $q\ge q_0$ the family 
of Dirichlet $L$-series $L(s, \chi)$, with $\chi$ any odd character modulo $q$, has a Siegel zero $\beta_0$ then
$$
\max\{R(q)e^{E_1(1- \beta_0)},R(q)^{-1}e^{-E_1(1- \beta_0)}\} 
< e^{0.41}
\, (\log q)^2 \, \ell(q).
$$
\item
\label{under-rh-odd}
If the Riemann Hypothesis holds 
for every Dirichlet $L$-series $L(s,\chi)$, with $\chi$ an odd character modulo $q$ for 
some prime $q\ge q_1$, then
$$
\max\{R(q),R(q)^{-1}\} <  
e^{0.41}
\log q.$$
\end{enumerate}
\end{Thm}

Notice that
$\max\{R(q),R(q)^{-1}\}=e^{|r(q)|}.$
Several other comments are in order. Here we
state the most relevant ones, and refer to
Section \ref{sec:thm1-comments} for further ones.
\begin{Rem}
Note that when compared with the no Siegel zero situation, assuming 
the Riemann Hypothesis for the Dirichlet $L$-series attached to odd
characters only allows one to remove a factor 
that tends to infinity arbitrarily slowly.
\end{Rem}

\begin{Rem}\label{rmk41}
The value  $0.41$ in 
Theorem \ref{rq-direct} can be further sharpened to $0.39$
by arguing as in Remark \ref{sharp-constant} below.
\end{Rem}
\begin{Rem} 
\label{E1bound}
We have
$1 \ll E_1(1- \beta_0) < \eps \log q + c(\eps)$,
where  $c(\eps)$ is ineffective
(see Section \ref{sec:thm1-comments} for a proof). 
\end{Rem}

The reader might wonder how close
Theorem \ref{rq-direct} is to the truth.
Sharp estimates were conjectured by
Granville \cite[\S~9]{Gr}, who speculated
that for $\epsilon>0$ and $q$ large enough, we have
\begin{equation*}
\max\{R(q),R(q)^{-1}\}  <
(\log_2 q)^{\frac{1}{2}+\epsilon},
\end{equation*}
with this result being false if $\frac12$ is being replaced 
by any smaller number (where here and in the sequel $\log_2 q$ denotes $\log\log q$).
The optimum is related to strong failure of prime number equidistribution, an extremely rare situation that is very far from being understood.

The paper is organised as follows:
In Section \ref{sec:prelim} 
we recall results we need (mainly from prime number theory)
and in Section \ref{sec:Kummer} we present
the Kummer ratio conjecture and we prove
an explicit constant version of a classical result
by Ankeny and Chowla (Lemma \ref{expl-new}).
In Section \ref{sec:proof-thm-rq-direct}
we prove Theorem \ref{rq-direct}.
Many comments and remarks about comparing Theorem \ref{rq-direct}
with similar results in the literature are collected in Section \ref{sec:thm1-comments}.
Section \ref{sec:numerical} 
expands on our earlier preprint \cite{earlierwork}.
It provides an efficient algorithm to compute $R(q)$
and some numerical data and graphical representations
regarding the distribution of $r(q)$ that might be the starting
point for future works.
For example, the presence of 
secondary ``spikes'' close to $\pm\frac{1}{4}$ and $\pm\frac{1}{8}$ demonstrates in a
beautiful way the contributions of the primes $q$ for which $2q\pm1$, or $4q\pm1$,
are prime too.

\section{Preliminaries}
\label{sec:prelim}
\subsection{Prime number distribution}
In this section, we recall the material we need on the distribution of prime numbers, using the notations
\begin{equation*}
\pi(t) = \sum_{p \le t} 1, \quad \quad
\pi(t;d,a) = \sum_{\substack{p \le t \\ p \equiv a  \pmod*{d}}} 1,
\end{equation*}
\begin{equation*}
\theta(t) = \sum_{p \le t} \log p, \quad \quad
\theta(t;d,a) = \sum_{\substack{p \le t \\ p \equiv a  \pmod*{d}}} \log p,
\end{equation*}
and
\begin{equation*}
\psi(t) = \sum_{n \le t} \Lambda(n), \quad \quad
\psi(t;d,a) = \sum_{\substack{n \le t \\ n \equiv a  \pmod*{d}}} \Lambda(n),   
\end{equation*}
where $\Lambda$ denotes the von Mangoldt function.
For fixed coprime 
integers $a$ and $d$, we have asymptotic equidistribution: 
$$
\pi(t;d,a)\sim \frac{\pi(t)}{\varphi(d)},
\quad
\theta(t;d,a)\sim \frac{\theta(t)}{\varphi(d)}
\quad
\textrm{and}
\quad
\psi(t;d,a)\sim \frac{\psi(t)}{\varphi(d)}
\quad (t\to \infty),
$$
with $\varphi$ Euler's totient.
While the asymptotics for the quantities above is available only for small $d$, 
say, $d \le (\log t)^{A}$, $A>0$, the following celebrated result 
concerns the accuracy of the first approximation on average when 
the moduli $d$ are allowed to be large with respect to $t$.
For every  $A > 0$, we have
the bound
\begin{equation} 
\label{BV-type-ineq}
\sum_{d \le {\mathcal Q}}
\max_{t \le u} 
\max_{(a,d)=1}
\Bigl\vert 
\psi(t;d,a)-\frac{\psi(t)}{\varphi(d)}
\Bigr \vert 
\ll 
\frac{u}{(\log u)^{A}},
\end{equation}
where ${\mathcal Q}={\mathcal Q}(u)$ is a suitable function,
and the implicit constant may depend on $A$ and ${\mathcal Q}$.
The upper bound \eqref{BV-type-ineq},  with 
${\mathcal Q}(u) = \sqrt{u}/(\log u)^B$, $B = B(A)>0$,
was independently proved by Bombieri 
and A. I. Vinogradov in 1965, see
\cite[\S~9.2] {CojocaruM2006}.

A similar statement for 
$t=u$ and ${\mathcal Q}(u) = u^{1-\eps}$ with $0<\eps<1$,
is unproved yet and commonly called
the \emph{Elliott-Halberstam conjecture}, see 
\cite{CojocaruM2006,ElliottH1968/69}. 

\begin{conj}
[Elliott-Halberstam]
\label{EHconjecture}
For every $\eps>0$ and $A>0$,
\[
\sum_{q\le u^{1-\eps}} \max_{(a,q)=1}
\Bigl|\pi (u;q,a)-{\frac {\pi (u)}{\varphi (q)}}\Bigr| 
\ll_{A,\eps} 
\frac{u}{(\log u)^A}.
\]
\end{conj} 
Statements equivalent to \eqref{BV-type-ineq} and the
Elliott-Halberstam conjecture with the $\psi(t;q,a), \psi(t)$-functions 
replaced by the $\pi(t;q,a), \pi(t)$-functions, or the $\theta(t;q,a), \theta(t)$-ones, 
can be easily obtained via partial summation. 

An important tool we will use is the following theorem. 
\begin{CThm}[Brun-Titchmarsh]
\label{BT-thm}
Let $x,y>0$ and $a,q$ be positive integers such that $(a,q)=1$.
Then
\begin{equation}
\label{BT-estim}
\pi(x+y;q,a) - \pi(x;q,a) < \frac{2y}{\varphi(q) \log(y/q)},
\end{equation}
for all $y >q$.
\end{CThm}
For a proof, see, e.g., Montgomery-Vaughan \cite[Theorem~2]{MVsieve}.

\subsection{Siegel zeros}
\label{sec:Siegel-zero}
Let $K \ne \mathbb{Q}$ be an algebraic number field having $d_K$ as its absolute discriminant over the rational numbers. Then, see
Stark \cite[Lemma~3]{stark}, $\zeta_K(s)$ has at most one zero in the region in the complex plane determined by 
$$ 
\Re(s) 
\geq
1 - \frac{1}{4 \log d_K},
\quad \quad
|\Im(s)| \leq \frac{1}{4 \log d_K}.
$$
If such a zero exists, it is
real, simple and often called 
\emph{Siegel zero}. 
When $K = \mathbb{Q}(\zeta_q)$,
using \eqref{oud} it is easy to see that the Siegel zero is attached to the family of
Dirichlet $L$-series $\pmod*{q}$. In this case, the Dirichlet character $\chi$
such that $L(s,\chi)$ has the Siegel zero is called the \emph{exceptional character}
and it is a well known fact that it is quadratic.

We will also use the Riemann Hypothesis (RH$_{\textrm{odd}}(q)$) for the Dirichlet $L$-series
attached to odd Dirichlet characters.
\begin{conj}[RH$_{\textrm{odd}}(q)$]
\label{RH-odd-conjecture}
Let $q$ be an odd prime.
The non-trivial zeros of the Dirichlet $L$-series $L(s,\chi)$, 
where $\chi$ runs over the set of the odd Dirichlet characters $\pmod*{q}$, 
are on the line $\Re(s)=\frac{1}{2}$.
\end{conj}
 
\subsection{Admissible sets of large measure}
\par Let $\mathcal A=\{a_1,\ldots,a_s\}$ be a set
of $s$ distinct natural numbers.
We define the measure
\[
\mu(\mathcal A)=\sum_{k=1}^s\frac{1}{a_k}.
\]
Given a prime $p$, let 
$\omega(p)$ denote the number of solutions modulo 
$p$ of the equation
\begin{equation}
\label{omegaeq}
X\prod_{i=1}^s(a_iX+1)\equiv 0 \pmod*{p}.
\end{equation}
A set $\mathcal A$ is said to be \emph{admissible} if 
$\omega(p)<p$ for every prime $p$. As $\omega(p)\le s+1$,
we see that $\mathcal A$ is admissible, if and only if 
$\omega(p)<p$ for every prime $p\le s+1$. 
We observe that if we change every term $a_iX+1$ by
$a_iX-1$ in \eqref{omegaeq}, the number of solutions
is also still $\omega(p)$. 
\par The admissible sets
are relevant for determining which sets of linear forms can
(presumably) be infinitely often all simultaneously prime.

\begin{conj}[Hardy-Littlewood \cite{HardyL1923}, lower bound version]
\label{HLconjecture}
Suppose $\AA = \{a_1,\ldots,$ $a_s\}$ is an admissible set.
Choose $b\in \{-1,1\}$.
Then the number of integers $n\le x$ such that the integers 
$n,a_1n+b,\ldots,a_sn+b$ are all prime is of cardinality $\gg_{\mathcal A} x/(\log x)^{s+1}$.
\end{conj}
 
\begin{conj}[Hardy-Littlewood \cite{HardyL1923} for Sophie Germain primes]
\label{HLsophie}
There are $ \gg x/(\log x)^2$ primes $ p \leq x$ for which $2p+1$ is also prime.
\end{conj}

\par The following result, conjectured 
by Erd\H{o}s (1988), shows that there are admissible
sets having  arbitrarily large measure $\mu$.

\begin{Thm*}[Granville \cite{Gr}]
There is a sequence of admissible sets $\mathcal A_1,\mathcal A_2,\ldots$ 
such that $\lim_{j\rightarrow \infty}$ $\mu(\mathcal A_j)=\infty$. 
We have $\overline{\mathcal M} = [0,\infty]$, with $\overline{\mathcal M}$ 
the closure of the set $ \{ \mu(\mathcal{A}) : \mathcal{A} \text{ is ~ an ~ admissible ~ set}\}$. 
\end{Thm*}

\section{Connecting Kummer's ratio conjecture with prime power sums}
\label{sec:Kummer}

\par The orthogonality property of odd characters 
$$
\frac{2}{q-1} \sum_{\chi(-1)=-1} \chi(a) \overline{\chi}(b) = 
\begin{cases}
\pm 1, & b \equiv \pm a  \pmod*{q},\\
0, &\text{otherwise},
\end{cases}
$$
gives us
$$\sum_{\chi(-1)=-1} \log(L(s,\chi)) = \frac{q-1}{2}\lim_{x\rightarrow \infty}
\Big(
\sum_{\substack{m\ge 1 ;\ p^m\le x\\ p^m\equiv 1 \pmod*{q}}}\!\! \frac{1}{mp^{ms}}- 
\sum_{\substack{m\ge 1 ;\ p^m\le x \\ p^m\equiv -1 \pmod*{q}}}\!\! \frac{1}{mp^{ms}}\Big).
$$
In the papers \cite{deb,MasMon,Puchta} just mentioned, the authors consider the latter function
in a neigborhood of $s=1$ (but whereas
Debaene \cite{deb} and Puchta \cite{Puchta} take higher derivatives into account,
Masley and Montgomery \cite{MasMon} stopped at the first derivative). Here, we will actually set $s=1$,
which in combination with \eqref{hasse} yields
\begin{equation}
\label{loggie}
r(q)=\log R(q) =\frac{q-1}{2}\lim_{x\rightarrow \infty}
\Big(\sum_{\substack{m\ge 1 ;\ p^m\le x\\ p^m\equiv 1 \pmod*{q}}}\!\! \frac{1}{m p^m}-
\sum_{\substack{m\ge 1; \ p^m\le x \\ p^m\equiv -1 \pmod*{q}}}\!\! \frac{1}{mp^m}\Big).
\end{equation}
\begin{Defi}
The argument in the limit we denote by $f_q(x)$ and $f_q:=\lim_{x\rightarrow \infty}f_q(x)$.
\end{Defi}
Note that Kummer's conjecture is equivalent with $f_q=o(1/q)$ as $q$ tends to infinity.
The idea is now to choose $x$ as small 
as possible so that the resulting error 
$f_q-f_q(x)$ is still reasonable. In 
attempting to decrease $x$, the Bombieri-Vinogradov theorem and 
the Brun-Titchmarsh inequality play a crucial role.
The main contribution to $f_q(x)$ comes from the
term with $m=1$, and is denoted by $g_q(x)$:
\begin{equation}
\label{geeq}
g_q(x):=\sum_{\substack{p\le x\\ p\equiv 1 \pmod*{q}}}\!\! \frac{1}{p}-
\sum_{\substack{p\le x \\ p\equiv -1 \pmod*{q}}}\!\! \frac{1}{p}.
\end{equation}
In the following, we will also use the notation
\begin{equation}
\label{Sbx-def}
S_q(b,x):=\sum_{\substack{p\le x\\ p\equiv b \pmod*{q}}}\!\! \frac{1}{p},
\end{equation}
where $b\in\{-1,1\}$,
so that $g_q(x)=S_q(1,x)-S_q(-1,x)$.

Taking all this into account
Granville \cite{Gr} showed that if Kummer's conjecture is true then for every $\delta>0$ we must have
\begin{equation*}
g_q(q^{1+\delta})=o\Bigl(\frac{1}{q}\Bigr),
\end{equation*}
for all but at most $2x/(\log x)^{3}$ primes $q\le x$. He used this to show that
$c^{-1}\le R(q)\le c$
for a positive proportion $\rho(c)$ of primes $q\le x$, where $\rho(c)\rightarrow 
1$ as $c$ tends to infinity. 
\par Murty and Petridis \cite{MP} improved this as follows.
\begin{Thm*}
There exists a constant $c>1$ such that for a sequence of primes with natural density 
1 we have $$ \max\{R(q),R(q)^{-1}\}\le c.$$ 
If the Elliott-Halberstam conjecture (Conjecture \ref{EHconjecture}) is true, then
we can take $c=1+\epsilon$ for any fixed $\epsilon>0$.
\end{Thm*}
Although typically $R(q)$ is close to $1$, conjecturally very different behavior 
also occurs 
(on very thin sets of primes).
\begin{Thm*}[Granville \protect{\cite[Theorems~2~and~4]{Gr}}]
If the lower bound version of the Hardy-Littlewood conjecture (Conjecture \ref{HLconjecture})
is true, and also the Elliott-Halberstam conjecture (Conjecture \ref{EHconjecture}), 
then for any admissible set  ${\mathcal A}$, the numbers
$e^{\mu(\mathcal A)/2}$ and $e^{-\mu(\mathcal A)/2}$ are both
limit points of the sequence $\{R(q):q\text{~is~prime}\}$.
Furthermore, this sequence has $[0,\infty]$ as set of limit points.
\end{Thm*}
\begin{cor*}
Under the above conjectures, Kummer's ratio conjecture \ref{Kconjecture} is false.
\end{cor*}
Actually, on taking ${\mathcal A}=\{2\}$, 
it suffices to assume the Hardy-Littlewood conjecture for Sophie German primes 
(Conjecture \ref{HLsophie}) instead of the full Hardy-Littlewood conjecture (Conjecture \ref{HLconjecture}), and Granville proved that 
$\min\{R(q),R(q)^{-1}\} >e^{0.249}$ for $\gg x/(\log x)^2$ primes $q\le x$. 
Likewise, other cases where $\mathcal A$ contains only one element produce a relatively 
thick set of $R(q)$ outliers, something also our numerics show. The value distribution of $r(q)$ 
was studied in detail by Croot and Granville \cite{CG}.

One can also wonder how large $R(q)$ can be as a function of $q$, an issue that we discussed just after Remark \ref{51}.

\subsection{A  useful lemma}
The following lemma, inspired by a result of  Ankeny-Chowla, see the estimate of $C_4$ in \cite{AC}
and \cite[Lemmas 1 and 2]{Gr},
will be a crucial ingredient in the proof of Theorem \ref{rq-direct}.
Let
\begin{equation}\label{mad}
t_q:=
\sum_{\substack{m\ge 2 \\ p^m\equiv 1 \pmod*{q}}}\!\! \frac{1}{m p^m}-
\sum_{\substack{m\ge 2 \\ p^m\equiv -1 \pmod*{q}}}\!\! \frac{1}{mp^m}.
\end{equation}
By \eqref{loggie}-\eqref{geeq} and \eqref{mad} we have
\begin{equation*}
r(q)=\frac{q-1}2\bigl(t_q+\lim_{x\rightarrow \infty}g_q(x)\bigr).
\end{equation*}
\begin{lem}
\label{expl-new} 
There exists a constant $c>0$ such that for every odd prime $q$ we have
\begin{equation*}
\vert t_q \vert < \frac{1}{q}\Bigl( \frac{43}{13}-\frac{18}{13}\zeta(3)\Bigr)
+ \frac{c}{q\log q}.
\end{equation*}
\end{lem}

\begin{proof}
Given any integer $b$ we put
\begin{equation}
\label{Sb-def}
S_q(b):=\sum_{\substack{m\ge 2 \\ p^m\equiv b \pmod*{q}}}\!\! \frac{1}{m p^m}.
\end{equation}
In the rest of the proof we will assume that $b\in \{-1,1\}$.
By \eqref{mad} we have $t_q = S_q(1)-S_q(-1)$, and using
$S_q(b)>0$, we
obtain $-S_q(-1) \le t_q \le S_q(1)$. Thus,
\begin{equation}
\label{main-lemma1}
\vert t_q \vert \le 
\max \{S_q(1), S_q(-1)\}.
\end{equation}
We split $S_q(b)$ into three subsums
according to whether $p^m \leq q  (\log q)^{2}$, $q(\log q)^{2} < p^m < q^2$
or $p^m>q^2$.
The contribution to the final estimate of the sums over the first range
 will be less than $c_1/q$, 
with a constant $c_1>0$ that will be explicitly determined, while the
others contribute $\ll 1/(q\log q)$.

We first consider the case $p^m>q^2$. 
For any prime $p>q\ge 3$ we have
\[
\sum_{ m \geq 2 } \frac{1}{m p^m}
\le 
\frac{1}{2} \frac{1}{p(p-1)} 
<\frac{1}{p^2}.
\]
Moreover, for $p\le q$, the condition $p^m > q^2$ implies
$m\ge 3$ and hence 
\[
\sum_{\substack{ m \geq 2;\, p\le q \\ p^m > q^2}} \frac{1}{m p^m}
\le
\frac{1}{3}\sum_{\substack{ m \geq 3;\, p\le q\\ p^m > q^2}} \frac{1}{p^m}
=
\frac{1}{3p^3} \sum_{j\ge 0}\frac{1}{p^j}
\le
\frac{1}{3q^2} \frac{p}{p-1} 
<\frac{1}{q^2}.
\]
Combining these estimates 
we arrive at
\begin{align}
\notag
\sum_{\substack{m \geq 2 ;\ p^m>q^2 \\ p^{m} \equiv b \pmod*{q}}}\frac{1}{m p^m}
& \leq 
 \sum_{p > q} \sum_{m \geq 2} \frac{1}{m p^m} +
 \sum_{p \le q} \sum_{\substack{ m \geq 2 \\ p^m > q^2}} \frac{1}{m p^m}
<\sum_{p > q} \frac{1}{p^2} +  \sum_{p \le q}  \frac{1}{q^2}
\\
\label{inter-new}
& = 2\int_q^{\infty} \frac{\pi(t)}{t^3} dt
\ll
\int_q^{\infty} \frac{dt}{t^2 \log t} 
\ll
\frac{1}{q \log q},
\end{align}
where we used the 
partial summation
formula
and  the Chebyshev bound in the weaker form $\pi(t) \ll t/\log t$.
{}From the proof of \cite[Lemma 1]{Gr} we obtain
\begin{equation}
\label{Granville-lemma1}
\sum_{\substack{ m \geq 2 ;\ q(\log q)^{2} < p^m < q^2 \\ p^{m} \equiv b \pmod*{q}}}
\frac{1}{m p^m}
\le 
\sum_{m=2}^{\lfloor 4 \log q \rfloor} \frac{1}{m} \frac{2m}{q (\log q)^2}
\ll \frac{1}{q \log q}.
\end{equation}

We now proceed as in the proof of \cite[Lemma 2]{Gr}.
We are left in \eqref{Sb-def} with the cases $p^m\equiv b \pmod*{q}$, $m\ge 2$ and $p^m \leq q  (\log q)^{2}$. 
There are
at most $m$ values of $p$ for which $p^m\equiv b \pmod*{q}$ and $p^m \leq q  (\log q)^{2}$,
so each sum can be maximized by assuming that 
$q+b$ and $2q+b$ are squares,
$3q+b$, $4q+b$ and $5q+b$ are cubes, and so on.
Setting 
$$\alpha(m):= \frac{1}{2}(m^2 - m),\quad \beta(m):=\frac{1}{2}(m^2 + m) - 1,$$ 
we obtain
\begin{align}
\notag
\sum_{\substack{m \geq 2; \, p^m \leq q  (\log q)^{2} \\ p^m\equiv b \pmod*{q}}} &
\frac{1}{m p^m}
\le
\sum_{m\geq 2} \frac{1}{m} \!\sum_{r = \alpha(m)}^{\beta(m)} 
 \frac{1}{ rq +b} 
 \le
\sum_{m\geq 2} \frac{1}{m} \!\sum_{r = \alpha(m)}^{\beta(m)} 
 \frac{1}{rq -1} 
\\&
\notag
=
\sum_{m\geq 2} \frac{1}{m}  \!\sum_{r = \alpha(m)}^{\beta(m)} 
\frac{1}{rq}   \Bigl(1+ \frac{1}{rq -1}\Bigr)
\\&
\le
\label{isolate}
\frac{1}{q} 
\Bigl(\frac{3}{4}+ \frac{1}{2q -2} + \frac{1}{8q -4} \Bigr)
+
\frac{1}{q} 
\Bigl(1+ \frac{1}{3q -1} \Bigr)
\sum_{m\geq 3} \frac{1}{m}  \sum_{r = \alpha(m)}^{\beta(m)} 
\frac{1}{r},
\end{align}
in which we have isolated the contribution of the terms corresponding to $m=2$.
The innermost sum in \eqref{isolate} can be bounded as follows:
\begin{align}
\notag
 \sum_{r = \alpha(m)}^{\beta(m)} 
\frac{1}{r} 
&\le 
\frac{2}{m^2-m} + \int_{\alpha(m)}^{\beta(m)} \frac{du}{u}
= \frac{2}{m^2-m} + \log \Bigl(1+ \frac{2}{m} \Bigr)
\\
\notag
&
\le 
\frac{2}{m^2-m} + \frac{2}{m} \frac{3m+1}{3m+4}
= \frac{2}{m} \Bigl(\frac{m}{m-1} -\frac{3}{3m+4}\Bigr)
\le
\frac{2}{m-1} - \frac{18}{13m^2},
\end{align}
in which we used the inequalities
$\log(1+x) \le \frac{x}{2}\cdot \frac{x+6}{2x+3}$
for every $x\ge 0$,
see \cite[eq.~(22)]{Topsoe2007},
and $3m+4 \le \frac{13}{3}m$, which holds for every $m\ge 3$. 
Hence
\begin{equation}
\label{innermost1}
\sum_{m\geq 3} \frac{1}{m} 
\sum_{r = \alpha(m)}^{\beta(m)}  
\frac{1}{r}
\le
\sum_{m\geq 3} \frac{2}{m(m-1)}  
- \frac{18}{13}\Bigl(\zeta(3) - \frac{9}{8}\Bigr) 
=
\frac{133}{52}- \frac{18}{13} \zeta(3).
\end{equation}
Inserting the bound for the double sum in \eqref{innermost1} into \eqref{isolate} we obtain
\begin{equation}
\label{ahna-new}
\sum_{\substack{m \geq 2; \, p^m \leq q  (\log q)^{2} \\ p^m\equiv b \pmod*{q}}} 
\frac{1}{m p^m}
\le 
\frac{1}{q} \Bigl(\frac{43}{13}-\frac{18}{13}\zeta(3) + \frac{c_2}{q}\Bigr),
\end{equation}
where $c_2>0$ is a suitable constant.
Inserting \eqref{inter-new}-\eqref{Granville-lemma1} and \eqref{ahna-new}
into \eqref{Sb-def} we deduce that there exists a constant $c>0$ such that
\begin{equation}
\label{Sb-estim}
S_q(b) < \frac{1}{q}\Bigl( \frac{43}{13}-\frac{18}{13}\zeta(3)\Bigr)
+ \frac{c}{q\log q}.
\end{equation}
The result immediately follows from \eqref{Sb-estim} and \eqref{main-lemma1}.
\end{proof}

\begin{Rem}
Lemma \ref{expl-new} allows one to improve the estimate $\Sigma_2=\frac{q-1}{2}t_q$ in
\cite{zhangyixi}\footnote{Also \cite[Lemma~2.2]{zhangyixi} makes use of  
\cite[Lemma~10]{LuZhang} in which a term $-\log_2 2$ is missing, see also
footnote \ref{footnote-luzhang} on page \pageref{footnote-luzhang}.}, 
and hence to improve some of the results there.
\end{Rem}

\begin{Rem}
\label{sharp-constant}
The leading constant in Lemma
\ref{expl-new}
can be further improved by isolating  
more initial terms of the sum over $m$ in equation \eqref{isolate}; 
in fact, with the aid of a computer program, it is not hard to see that
isolating the contribution of the first $10$ values of $m$ (corresponding with $r=1,\dotsc,54$),
the value of this constant can be reduced from 
$\frac{43}{13}-\frac{18}{13}\zeta(3)=1.64330\dots$
to $1.60091\ldots<1.601$.
Moreover, again using a computer program, it is also possible to directly compute 
$\sum_{m=2}^{T} \frac{1}{m} 
\sum_{r = \alpha(m)}^{\beta(m)} 
\frac{1}{r}\bigl(1+ \frac{1}{ rq -1} \bigr)$
with $T=2000$ (so $r=1,\dotsc, 2\,000\,999$). 
This leads to a value $>1.59908$, and so the upper bound $1.601$ is almost optimal.
\end{Rem}

\section{Proof of Theorem \ref{rq-direct}}
\label{sec:proof-thm-rq-direct}

Using \eqref{hasse}, the Euler product for $L(1,\chi)$, $\chi\ne \chi_0$,
we obtain
\begin{equation}
\label{starting}
r(q)
= 
- \sum_{\chi(-1)=-1} \sum_p \log\Bigl(1-\frac{\chi(p)}{p}\Bigr)
=
\sum_{\chi(-1)=-1} \sum_p \sum_{m\ge1} \frac{\chi(p^m)}{mp^m}
= 
\Sigma_1 + \Sigma_2,
\end{equation}
say, where $\Sigma_1$ is the contribution of the primes ($m=1$)
and $\Sigma_2$ that of the prime powers ($m\ge 2$).
We have, see the introduction of Section \ref{sec:Kummer},
$$\Sigma_1=\sum_{\chi(-1)=-1} \sum_{p}  \frac{\chi(p)}{p}=\frac{q-1}{2}\lim_{x\rightarrow \infty}g_q(x),$$
and
$$\Sigma_2=\frac{q-1}{2}
\Big(\sum_{\substack{m\ge 2 \\ p^m\equiv 1 \pmod*{q}}}\!\! \frac{1}{m p^m}-
\sum_{\substack{m\ge 2 \\ p^m\equiv -1 \pmod*{q}}}\!\! \frac{1}{mp^m}\Big)=\frac{q-1}2 t_q.
$$
Lemma \ref{expl-new} then yields 
\begin{equation}
\label{Sigma2-estim}
\vert \Sigma_2 \vert
< \frac{43}{26}-\frac{9}{13}\zeta(3) + \frac{c}{\log q},
\end{equation}
where $c$ is a positive constant.
We will use this inequality in the proofs of all the three parts of this theorem.

{}From now on, we will assume that $b\in \{-1,1\}$ and that $q$
is a sufficiently large prime.
We now consider $\Sigma_1$; we split the prime sum into the 
ranges $p\le x_1$ and $p>x_1$ (where $x_1$
will be chosen later on), and proceed to estimate these subsums.
The first ingredient is supplied by the following inequalities valid for any $x>0$:
\begin{equation}
\label{bounds-sigma1-trunc}
- \frac{q-1}{2} S_q(-1,x)
\leq  \sum_{\chi(-1)=-1} \sum_{p \leq x}  \frac{\chi(p)}{p} 
\leq \frac{q-1}{2}S_q(1,x),
\end{equation}
where $S_q(b,x)$ is defined in \eqref{Sbx-def}.
Using $S_q(b,x)>0$, from \eqref{bounds-sigma1-trunc} we obtain 
\begin{equation}
\label{main-thm1}
\Bigl\vert  \sum_{\chi(-1)=-1} \sum_{p \leq x}  \frac{\chi(p)}{p} \Bigr\vert 
\le 
\frac{q-1}{2} \max \{S_q(1,x), S_q(-1,x)\}.
\end{equation}

We now proceed to estimate $S_q(b,x)$.
Letting $x \geq q^2$, by
partial summation and the Brun-Titchmarsh theorem,
see Classical Theorem \ref{BT-thm}, we have
\begin{align}
\notag
\frac{q-1}{2} \sum_{\substack{kq < p \leq x\\ p \equiv b \pmod*{q}} } \frac{1}{p}
&= \frac{q-1}{2} \Bigl(\frac{\pi(x;q, b)}{x} - \frac{\pi(kq; q,b)}{kq} 
+ \int_{kq}^{x} \frac{\pi(u; q,b)}{u^2} du\Bigr) 
\\
\notag
&
\le \frac{q-1}{2}
\Bigl(\frac{2}{(q-1) \log(x/q)} 
+
\frac{2}{(q-1)}
\int_{kq}^{x} \frac{du}{u \log(u/q)}
\Bigr)
\\
\label{large-primes}
&\leq  
\log_2 \bigl( \frac{x}{q} \bigr)
- \log_2 k +\frac{1}{\log q},
\end{align}
where $k\ge 3$ is an odd integer we will choose later.
Since $k$ is odd and $q\ge 3$, any prime $p \leq kq$
with $p \equiv b \pmod*{q}$ 
is of the form $p = 2jq + b$, with $j\in \{1,\dotsc,\frac{k-1}{2}\}$.
So 
\begin{equation}
\label{small-primes}
 \frac{q-1}{2} \sum_{\substack{p \leq kq \\ p \equiv b \pmod*{q}} } \frac{1}{p}  
 \le
  \frac{1}{2}   \sum_{j=1}^{(k-1)/2}  \frac{q-1}{2jq-1}
< \frac{1}{4} \sum_{j=1}^{(k-1)/2} \frac{1}{j}
= \frac{1}{4} H_{\frac{k-1}{2}},
\end{equation}
with $H_{n} := \sum_{j=1}^{n}\frac{1}{j}$ is the $n$-th harmonic number.
Combining \eqref{large-primes}-\eqref{small-primes} we obtain
\begin{equation} \label{127} 
\frac{q-1}{2}  S_q(b,x)
< \log_2 \bigl( \frac{x}{q}\bigr) + c_1(k) + \frac{1}{\log q}  ,
\end{equation}
where $c_1(k):=\frac{1}{4} H_{\frac{k-1}{2}}  - \log_2 k$.
We now choose $k$ such that $c_1(k)$ is minimal.
It is not hard to see that this happens for $k=55$ and that $C_1:=c_1(55)  < -0.4152617906$.

Let $\ell(q)$ be a function that tends to infinity arbitrarily slowly and monotonically as $q$ tends to infinity.
Taking 
$x = x_1:= q^{\ell(q)}$
and $k=55$ in \eqref{127}, we have
\begin{align} \label{122}
\frac{q-1}{2}  S(b,x_1)
& < \log_2 q 
+\log \ell(q)
+ C_1 + \frac{1}{\log q}.
\end{align}
Inserting  \eqref{122} into  \eqref{main-thm1},
we obtain\footnote{\label{footnote-luzhang}There 
is an oversight in the proof of \cite[Lemma~10]{LuZhang}
since the (positive) term $-\log_2 2$ is missing. The reasoning that leads to \eqref{124} is an 
amended and improved version of the argument of Lu-Zhang.}
\begin{equation}\label{124}
\Bigr| \sum_{\chi(-1)=-1} \sum_{p \leq x_1}  
\frac{\chi(p)}{p}
\Bigr|  
< \log_2 q 
+\log \ell(q)
+ C_1 + \frac{1}{\log q}.
\end{equation}
This inequality will be used in the proofs of the first two parts of this theorem.

If there is no odd character modulo $q$ such that $L(s, \chi)$ has a Siegel zero, then
in addition to the estimate \eqref{124} 
by \cite[Lemmas~1, 7 and~8]{LuZhang}\footnote{\label{boost-footnote}Note that 
\cite[Lemma~7]{LuZhang} holds for every $T,x_1$ such that 
$\lim_{q \to \infty} \log(qT) /\log x_1 = 0$. This  allows us to choose 
$T=q^4$ and $x_1= q^{\ell(q)}$, where $\ell(q)$ tends to infinity
arbitrarily slowly and monotonically as $q$ tends to 
infinity. The final error term in 
\cite[Lemma~8]{LuZhang} is then $\ll 1/\ell(q)=o(1)$ as $q$ tends to infinity.}
we also have
\begin{align} \label{123}
\Bigr| \sum_{\chi(-1)=-1} \sum_{p > x_1}  \frac{\chi(p)}{p} \Bigr| 
& \ll \frac{1}{\ell(q)}.
\end{align}
We remark that
\eqref{123} also holds if 
$q \equiv 1 \pmod*4$ since this implies that there does not exist 
any odd quadratic Dirichlet character modulo $q$ and hence 
its attached Dirichlet $L$-series has no Siegel zero.
Using equations \eqref{124} and \eqref{123}, we have
\begin{align}\label{125}
| \Sigma_1 |    
< \log_2 q + \log \ell(q)  +C_1 + o(1).
\end{align}
Combining \eqref{Sigma2-estim} and \eqref{125},
we obtain
\begin{align*}
| r(q) | &  
< \log_2 q +  \log \ell(q) 
+ C_1 + \frac{43}{26}-\frac{9}{13}\zeta(3)   + \frac{1}{1000} 
< \log_2 q +  \log \ell(q) + 0.41.
\end{align*}
This completes the proof of Part \ref{nosiegelzero}.

We now prove Part \ref{Siegelzero}.
The starting point is  again  \eqref{starting}, but we need a more accurate analysis of the
contribution of $\Sigma_1$.
To do so, the first step is to split the prime sum $\Sigma_1$ in three 
subsums $S_1, S_2, S_3$ defined according
to whether $p\le x_1$, $x_1 < p \le x_2$ or $p\ge x_2$, 
where $x_2 = e^q$ and
$x_1= q^{\ell(q)}$, with $\ell(q)$ being any function that 
tends to infinity arbitrarily slowly and monotonically as $q$ tends to infinity.

We already estimated $S_1$ in \eqref{124} and will proceed to estimate $S_3$.
By \cite[Lemma~1]{LuZhang}, we have
\begin{equation}
\label{C-estim}
S_3 \ll q^2 e^{-c_1 \sqrt{q}},
\end{equation}
where $c_1>0$ is an absolute constant.

For $S_2$ we follow the first part of the argument in \cite[Lemma~8]{LuZhang}.
Recall (see, e.g., \cite[Ch.~19]{Davenport}) that if $\chi$ is a non principal 
character modulo $q$ and $2 \leq T \leq x$, then
\begin{equation}
\label{explicit-formula}
\theta(x,\chi) 
:= \sum_{ p \leq x} \chi(p) \log p
= - \delta_{\beta_0}\ \frac{x^{\beta_0}}{\beta_0} - 
\sideset{}{'}\sum_{|\gamma| \leq T} \frac{x^{\rho}}{\rho} 
+ O\Bigl( \frac{x (\log qx)^2 }{T} + \sqrt{x}\Bigr),
\end{equation}
where $\delta_{\beta_0} = 1$ if the Siegel zero $\beta_0$ exists and  is zero otherwise, and
$\sum^\prime$ is the sum over all non-trivial zeros 
$\rho = \beta + i \gamma$ 
of $L(s,\chi)$,
with the exception of $\beta_0$ and 
its symmetric zero $1-\beta_0$.

Assume that there exists a Siegel zero $\beta_0$ with odd 
associated character. Then,
by the partial summation formula and \eqref{explicit-formula} with $T=q^4$, we have
\begin{align}\label{late}
\notag
S_2 := \sum_{\chi(-1) = -1} 
&\sum_{x_1 < p \le x_2} \frac{\chi(p)}{p} 
 = 
 \sum_{\chi(-1) = -1} 
 \Bigl( 
 \frac{\theta(x_2, \chi)}{ x_2 \log x_2} 
 - \frac{\theta(x_1, \chi)}{ x_1 \log x_1} 
 + \int_{x_1}^{x_2} \theta( u,\chi) \frac{1 +\log u}{ (u \log u)^2} du 
 \Bigr)\\
& = - \int_{x_1}^{x_2}  \frac{u^{\beta_0-2}}{ \log u}du - \int_{x_1}^{x_2}\Bigl( \sum_{\chi(-1) = -1} 
\sideset{}{'}\sum_{|\gamma| \leq q^4} u^{\rho - 2}\Bigr) \frac{du}{\log u} + \frac{q-1}{2} E_q,
\end{align}
where 
\begin{align}
\notag 
E_q
\ll \int_{x_1}^{x_2} \Bigl(\frac{(\log qu)^{2}}{q^4 u }  +  \frac{1}{u^{3/2}} \Bigr) \frac{du}{\log u} 
\ll \frac{1}{q^2}.
\end{align} 
By using \cite[Lemmas~7 and~8]{LuZhang}\footnote{See footnote \ref{boost-footnote} 
on page \pageref{boost-footnote}.}, we obtain that 
\begin{equation}
\label{Lemma7-LuZhang}
\int_{x_1}^{x_2}\Bigl( \sum_{\chi(-1) = -1} 
\sideset{}{'}\sum_{|\gamma| \leq q^4} u^{\rho - 2}\Bigr) \frac{du}{\log u}
\ll \frac{1}{\ell(q)}.
\end{equation}

We now proceed to evaluate the term depending on $\beta_0$ in \eqref{late}.
A direct computation using that 
$\log x_2 = q$ gives
\begin{equation*}
\int_{x_1}^{x_2}  \frac{u^{\beta_0-2}}{ \log u}du 
=  \int_{\log x_1}^{\log x_2} \frac{dt}{te^{(1-\beta_0)t}}
= E_1(1-\beta_0) 
- \int_{1-\beta_0}^{( 1-\beta_0)\log x_1} \frac{dt}{te^t}
- E_1( q(1-\beta_0)), 
\end{equation*}
where $E_1(u)$ denotes the exponential integral function.
Recalling that $x_1= q^{\ell(q)}$, 
where $\ell(q)$ tends to infinity arbitrarily slowly and monotonically as $q$ tends to 
infinity,
we  have 
\begin{equation}
\label{siegel-zero-term-tails}
\int_{1-\beta_0}^{( 1-\beta_0)\log x_1} \frac{dt}{te^t}
\le 
\log_2 x_1 =
\log_2 q + \log \ell(q)
\quad
\textrm{and}
\quad
E_1(q(1-\beta_0)) 
\ll \frac{1}{q}.
\end{equation}

Inserting \eqref{Lemma7-LuZhang}-\eqref{siegel-zero-term-tails} into  \eqref{late}, we finally obtain
\begin{equation}\label{f3}
\vert S_2 + E_1( 1-\beta_0) \vert
\le \log_2 q 
+ \log \ell(q)
+o(1).
\end{equation}
Combining \eqref{124}, \eqref{C-estim} and
\eqref{f3}, we obtain
\begin{equation}
\label{E1-eval}
\vert \Sigma_1 + E_1( 1-\beta_0)\vert 
< 2\log_2 q 
+ 2\log \ell(q)
+  C_1
+o(1).
\end{equation}
Recalling that  $\Sigma_2$ is estimated in \eqref{Sigma2-estim}, by
combining \eqref{starting}-\eqref{Sigma2-estim} and \eqref{E1-eval}, 
Part \ref{Siegelzero} follows.

It remains to prove Part \ref{under-rh-odd}. In this case 
we split $\Sigma_1$ in \eqref{starting} in two subsums $S_1, S_2$ 
defined according to $p\le x_1$ and $p>x_1$. 
Let $A>0$ be a constant to be chosen later and let $x = q^{2}(\log q)^A=:x_1$ in \eqref{127}.
We obtain
\begin{equation} \label{128}
|S_1| < \log_2 q  + C_1 +  (A+1) \frac{\log_2 q}{\log q},
\end{equation}
where $C_1:=c_1(55)  < -0.4152617906$.
Arguing as in \eqref{late}, by partial summation we  obtain
\begin{equation}\label{late-rh-odd}
S_2
 = \sum_{\chi(-1) = -1} 
 \lim_{y \to \infty} 
 \Bigl( \frac{\theta(y, \chi)}{y \log y} - \frac{\theta(x_1, \chi)}{ x_1 \log x_1} 
 + \int_{x_1}^{y} \theta( u,\chi) \frac{1 +\log u}{ (u \log u)^2} du \Bigr).
\end{equation}
Assuming Conjecture \ref{RH-odd-conjecture},
we have that $\psi(x,\chi):= \sum_{n\le x} \chi(n) \Lambda(n)
 \ll \sqrt{x} (\log x)^2$ for every odd Dirichlet character $\chi \pmod*{q}$,
see, e.g., \cite[p.~125, Ch.~20]{Davenport}.
Recalling $\psi(x,\chi) - \theta(x,\chi) \ll \sqrt{x}$, we have
that \eqref{late-rh-odd} becomes
\begin{equation}
\label{S2-rh-estim}
S_2
\ll 
\frac{q \log x_1}{\sqrt{x_1}}
\ll_A
(\log q)^{1-A/2}
= o(1),
\end{equation}
for every $A>2$.
Choosing $A=3$, by combining \eqref{128} and \eqref{S2-rh-estim} we have 
\begin{equation}
\label{sigma1-grh-estim1}
\vert \Sigma_1 \vert 
<\log_2 q  +C_1 +  \frac{c }{\sqrt{\log q}},
\end{equation}
where $c$ is a suitable positive constant.
Equations \eqref{Sigma2-estim} and \eqref{sigma1-grh-estim1} imply that
$$
|r(q)| < \log_2 q   + 0.41,
$$
and hence Part \ref{under-rh-odd} follows.
\qed

\section{Remarks on Theorem \ref{rq-direct}}
\label{sec:thm1-comments}
\noindent
In the introduction we already made some
comments on Theorem \ref{rq-direct}. Here we
make some further ones.\par

\medskip
\noindent \textbf{Comment to Remark} \ref{E1bound}.
It is easy to derive these 
estimates for $E_1(1- \beta_0)$.
We recall that 
\begin{equation}
\label{E1-series}
E_1(x) =
- \Euler -\log x  +\int_0^x (1-e^{-t})\frac{dt}{t}= 
- \Euler - \log x
- \sum_{k=1}^{\infty} \frac{(-x)^k}{(k!)k}\quad (x>0),
\end{equation}
where $\Euler$ is the Euler-Mascheroni constant.
Since 
$0<1-e^{-x}< x$, we infer from the first equality that
$$-\Euler - \log x < E_1(x) < -\Euler - \log x + x\quad (x>0).$$
On
using that for every $\eps>0$ there exists a constant $c_1(\eps)$ such that
$\beta_0 < 1- c_1(\eps)q^{-\eps}$,
the bounds for $E_1(x)$ lead to
$1 \ll E_1(1- \beta_0) < \eps \log q + c(\eps),$
where $c(\eps)$ is ineffective. 
Using the weaker,
but with an effective constant, estimate $\beta_0 < 1- cq^{-1/2}(\log q)^{-2}$ we obtain that 
$1 \ll E_1(1- \beta_0) < \frac{1}{2} \log q + 2\log_2 q + c_1$,
where $c_1>0$ is an effective constant.

\begin{Rem}[On the role of the Brun-Titchmarsh theorem in our results, I]
The power of the $\log q$-factor in the estimates for $R(q)$ in Theorem \ref{rq-direct} 
directly depends on 
\eqref{127}  
that follows from using the Brun-Titchmarsh theorem (Classical Theorem \ref{BT-thm})
in \eqref{large-primes}.
In particular, a key role in \eqref{large-primes} is played by the constant $2$ present in \eqref{BT-estim}; 
any improvement of this constant to, e.g., $2-\xi$,
$\xi\in(1,2)$, will lead to replace one $\log q$-factor with $(\log q)^{1-\xi/2}$ into Theorem \ref{rq-direct}.
From the works of Motohashi \cite{Motohashi1979}, Friedlander-Iwaniec
\cite{FriedlanderI1997}, Ramar\'e \cite[Theorem~6.5]{Ramare2009}  and Maynard
\cite[Proposition~3.5]{Maynard2013}, it is well known that replacing such a constant with any value less
than $2$ is equivalent with assuming there does not exist a Siegel zero 
for $\prod_{\chi\pmod*{q}}L(s,\chi)$.
Unfortunately, none of these results is applicable in our case since $R(q)$ 
depends on the odd Dirichlet characters only; hence, assuming, as in Part 
\ref{nosiegelzero} of Theorem \ref{rq-direct}, that the Dirichlet $L$-series associated 
to the odd characters do not have any Siegel zero is not enough to imply the possibility of using
\eqref{BT-estim} with a leading constant less than $2$.
\end{Rem}
\begin{Rem}[On the role of the Brun-Titchmarsh theorem in our results, II]
Montgomery and Vaughan, see \cite[Theorem~2]{MVsieve}, under the same hypotheses 
of Classical Theorem \ref{BT-thm}, also proved that there exists a
constant $C>0$, which they did not
make explicit, such that if $y>Cq$ then
\begin{equation}
\label{BT-alternative-MV}
\pi(x+y;q,a) - \pi(x;q,a) < \frac{2y}{\varphi(q) (\log(y/q) + \frac{5}{6})}.
\end{equation}
Usage of this estimate potentially allows
us to decrease the value of $C_1$ in Theorem \ref{rq-direct},
leading to improvements of the constants $0.41$ and $0.39$ 
(see also Remark \ref{rmk41}).
Once $C$ has been made explicit, \eqref{BT-estim} 
can be replaced by \eqref{BT-alternative-MV}.
Selberg \cite[vol.~2, p.~233]{SelbergCollected2}  obtained
\eqref{BT-alternative-MV}
with $2.8$ instead of 
$\frac{5}{6}$, but also did not make $C$ explicit.
\end{Rem}

\begin{Rem}
\label{51}
The first two parts of Theorem \ref{rq-direct}
sharpen Lu-Zhang \cite[Theorem~1]{LuZhang}
by reducing/enlarging the exponents of the $\log q$-factor 
from $\frac{7}{6}$ and $-\frac{4}{3}$ to, respectively, 
$1$ and $-1$. 
Again comparing with \cite[Theorem~1]{LuZhang}, 
in Part \ref{under-rh-odd} we assumed RH$_{\textrm{odd}}(q)$ 
(Conjecture \ref{RH-odd-conjecture}), instead of the Generalized Riemann 
Hypothesis.
Part \ref{Siegelzero} of Theorem \ref{rq-direct} improves 
the second part of Theorem~1.1 of Debaene \cite{deb} in two 
aspects:
the
term $-\log (1-\beta_0)$ is replaced
by a more accurate description involving 
$E_1(1-\beta_0)$, see \eqref{E1-series}, and 
the exponent of the  $\log q$-factor is reduced
from $4$ to $2$.
This reduction is a 
consequence of
sharpened estimates for the 
quantities $\Sigma_1,\Sigma_2$ used in the proof of Theorem \ref{rq-direct}.
\end{Rem}

\begin{Rem}
In Part  \ref{Siegelzero} 
of Theorem \ref{rq-direct} we included only the contribution
of  the Siegel zero $\beta_0$, and not 
of the zero $1-\beta_0$. We did this in order to be able to easily compare with the
result of Puchta \cite[Theorem~1]{Puchta} and Debaene \cite[Theorem~1.1]{deb}.
It is in fact pretty easy to obtain the contribution of the
zero $1- \beta_0$: arguing as in the proof of Theorem \ref{rq-direct}
we see that it equals
\begin{align} 
\notag
\int_{x_1}^{x_2}  \frac{u^{-1 -\beta_0}}{\log u}du & 
=  \int_{\log x_1}^{\log x_2} \frac{dt}{te^{\beta_0t}}
= E_1(\beta_0) 
- \int_{\beta_0}^{\beta_0\log x_1} \frac{dt}{te^t}
- E_1( q\beta_0), 
\end{align}
in which we also used the fact that $\log x_2 = q$.
Since $x_1= q^{\ell(q)}$ we also have 
\begin{equation*}
\int_{\beta_0}^{\beta_0\log x_1} \frac{dt}{te^t}
\le 
2e^{-\beta_0}
\le
2e^{-1/2}
\quad
\textrm{and}
\quad
E_1(q\beta_0) 
\ll \frac{1}{q}.
\end{equation*}
Hence, the main term in the previous formula is in the series expansion of
$E_{1}(\beta_0)$ and it is $-\log \beta_0$.
Adding the leading terms
of $E_{1}(\beta_0)$ and $E_{1}(1-\beta_0)$, we have that their total contribution
behaves as $-\log(\beta_0(1-\beta_0))$.
\end{Rem}

\section{An efficient algorithm to compute \texorpdfstring{$R(q)$}{Rq}}
\label{sec:numerical}

\newcommand{\pariversion}{PARI/GP (v.~2.15.4)}
\newcommand{\pythonversion}{Python (v.~3.11.6)}
\newcommand{\matplotlibversion}{Matplotlib (v.~3.8.2)}
\newcommand{\pandasversion}{Pandas (v.~2.1.3)}
\newcommand{\fftwversion}{FFTW (v.~3.3.10)}
\newcommand{\ubuntuversion}{Ubuntu~22.04.3~LTS} 
\newcommand{\clusteraddress}{\url{https://hpc.math.unipd.it}}
\newcommand{\capriaddress}{\url{https://capri.dei.unipd.it}}
\newcommand{\optiplexmachine}{Dell OptiPlex-3050, equipped with an Intel i5-7500 processor, 3.40GHz, 
32 GB of RAM}

\newcommand{\bound}{10^7}
 
Since $R(q)$  grows very quickly as $q$ increases, 
it is much better to evaluate $r(q)$  instead, and obtain $R(q)$ as $\exp(r(q))$.
We addressed this problem already in the preprint \cite{earlierwork}, and here present 
an updated version in which we obtain more accurate results.
For all these quantities we can use the algorithm
in \cite{Languasco2021a}, see also \cite{LanguascoR2021}, in which 
the Fast Fourier Transform (FFT) procedure is used to obtain the needed
values of $L (1,\chi)$ which are the main ingredients for getting $r(q)$.
The fundamental formula is 
the well-known relation,
see, e.g., eq.~(2.1) of Shokrollahi \cite{Shokrollahi1999}:
\begin{equation}
\label{first-factor-Bernoulli}
h_1(q) = 2q \prod_{\chi(-1) = -1} \Bigl(- \frac{B_{1,\chi}}{2} \Bigr),
\end{equation}
where $B_{1,\chi}$ is the first $\chi$-Bernoulli number
defined, see Proposition 9.5.12 of Cohen \cite{Cohen2007}, as 
\begin{equation}
\label{chi-Bernoulli-def}
B_{1,\chi}
:=
\sum_{a=1}^{q-1}  \frac{a}{q}  \chi(a) .
\end{equation}
Inserting
\eqref{chi-Bernoulli-def} and \eqref{first-factor-Bernoulli} 
into \eqref{Rq-def}, we obtain\footnote{The
term $1/q$ is taken on purpose outside the inner sum, 
as this helps in controlling the errors in the FFT procedure.}
\begin{equation*}
R(q)  
=
\Bigl(-\frac{\pi}{\sqrt{q}}\Bigr)^{\frac{q-1}{2}}  
\prod_{\chi(-1) = -1}\sum_{a=1}^{q-1}  \frac{a}{q} \chi(a).
\end{equation*}
We immediately have
\begin{equation}
\label{rq-start}
r(q)  
=  
\frac{q-1}{2} \Bigl( \log \pi - \frac{1}{2} \log q + i \pi\Bigr) 
+
\sum_{\chi(-1) = -1} 
\log\Bigl( \sum_{a=1}^{q-1}\frac{a}{q} \chi(a) \Bigr),
\end{equation}
where the second logarithm is the complex one.
Recalling that $R(q)>0$, it is clear  that 
$r(q)$ is a real number. Hence the 
imaginary part of the sum over the odd Dirichlet characters 
in \eqref{rq-start} must be
equal to $- \pi(q-1)/2$. We obtain 
\begin{equation}
\label{chi-Bernoulli-method-formula2}
r(q)
=  
\frac{q-1}{2} \Bigl( \log \pi - \frac{1}{2} \log q \Bigr) 
+ 
\sum_{\chi (-1) = -1}
\log \ \Bigl\vert
\sum_{a=1}^{q-1}  \frac{a}{q} \chi(a) 
\Bigl\vert.
\end{equation}  

As a possible alternative approach, one can start from 
\eqref{hasse} and use the fact that  
$L(1,\chi)
=
-
\frac{1}{q}
\sum_{a=1}^{q-1} \chi(a)
\digamma \bigl(\frac{a}{q}\bigr), 
$
where $\digamma(x)=(\Gamma^\prime/\Gamma)(x)$ is the \emph{digamma} function.
Very similar computations then lead to
\begin{equation}
\label{Hasse-method-formula2}
r(q)
= 
- 
\frac{(q-1)}{2} \log q 
+
\sum_{\chi (-1) = -1}  
\log \ \Bigl\vert
\sum_{a=1}^{q-1} \chi(a)
\digamma \bigl(\frac{a}{q}\bigr)
\Bigr\vert .
\end{equation}
In practice, though, it is better to use \eqref{chi-Bernoulli-method-formula2}
since no special function computations are needed there;
nevertheless, formula \eqref{Hasse-method-formula2} can be useful to double-check
our results.

The summation over $a$ in \eqref{chi-Bernoulli-method-formula2} 
can be handled using the FFT procedure, paying attention to choose
only the contributions of the odd Dirichlet characters, and we can also
embed here a \emph{decimation in frequency strategy}, see Section \ref{FFT-DIF}, 
as we already did in \cite{Languasco2021a} and \cite{LanguascoR2021}.

The FFT procedure requires $O(q)$ memory positions and
the computation of $r(q)$ via 
\eqref{chi-Bernoulli-method-formula2} has a computational cost of $O(q\log q)$ 
arithmetic operations plus the cost of computing  $(q-1)/2$ values of  the 
logarithm function and products:  so far, this is the fastest known algorithm 
to compute $r(q)$, see, for other less efficient algorithms, 
the works of Fung-Granville-Williams \cite{FungGW1992}, 
Shokrollahi \cite{Shokrollahi1999} and Jha \cite{Jha1995}.

Using this algorithm we were able to obtain more digits of the  
maximal and minimal champions for $R(q)$ with $q<10^{10}$, namely
\[
R(9697282541) =  1.7247411203\dotsc
\quad
\textrm{and}
\quad
R(116827429)  = 0.5756742526\dotsc,
\]
see also Tables \ref{table2}-\ref{table3}.
Further records for $q>10^{10}$, claimed to be correct up to $6$ decimals, 
were obtained in 2021 by Broadhurst \cite{Broadhurst2021}.
He directly evaluated the prime sums in formula \eqref{loggie} using $x$ 
between $10^{15}$ and $10^{21}$; then he 
gave a statistical estimate for the errors in such 
computations. 

Our approach to evaluate the computational error is more classical 
and it is presented in Section \ref{FFT-accuracy}.
We did not replicate Broadhurst's computations for $q>10^{10}$,
because we did not have enough computational resources at our disposal. 
The problem is the huge memory resources the FFT procedure would require in these cases. 
For example, the largest case we were able to handle, $q=9 854 964 401$,
already needed about 3TB of memory using FFTW \cite{FrigoJ2005}, a software library designed to
implement the FFT procedure, with the quadruple precision ($128$ bits) 
of the C programming language.

\subsection{Computations with trivial summing over \texorpdfstring{$a$}{a} (slower, more digits available).}

Unfortunately in PARI/GP \cite{PARI2023} and in \texttt{libpari} the FFT-functions work only if $q=2^t+1$, 
for some $t\in \N$.  So  we had to trivially perform these summations and  hence, in practice,  
this part is the most time consuming one as its computational cost is quadratic in $q$.
Nevertheless, this approach works nicely for small values of $q$ and is able
to provide the values of $R(q)$ which many decimals.

Being aware of such limitations, we performed the 
computation of $r(q)$ and $R(q)$  
for every prime $q$ in the range $3\le q\le 1000$, 
using a precision of $100$ decimals, see Table \ref{table1}. 
We also computed their values for 
$q=1451$, $2741$, $3331$, $4349$, $4391$, $5231$, $6101$, $6379$, $7219$, $8209$, $9049$, $9689$,
see the left part of Table \ref{table2}.
These numbers were chosen to  extend the available decimals for the 
known data (see Fung-Granville-Williams \cite{FungGW1992} and Shokrollahi \cite{Shokrollahi1999}). 
Likewise we also obtained them for 
$q=37189$, $42 611$, $149119$, $198 221$, $305 741$, $401179$;  see the right part of Table \ref{table2}.
For these values of $q$ it became clear that the
bulk of the computation time was spent on summing over $a$,
providing experimental evidence that replacing the trivial way of summing over $a$ 
with the FFT procedure is fundamental to be able to evaluate $r(q)$ and $R(q)$ for larger values of $q$.

\subsection{FFT-Decimation in frequency} 
\label{FFT-DIF} 
We give here a shortened version of the more general argument
presented in \cite[Section~8]{Languasco2023}.
The way to translate  eq.~\eqref{chi-Bernoulli-method-formula2} 
into a problem that can be handled using the FFT procedure is
based on the following remark.
Recalling that $q$ is prime, it is enough to determine a primitive root\footnote{We
recall that finding primitive roots is a computationally hard problem; but
we just need to do this once for each $q$ we will work with.}
$g$ of $q$, to represent the Dirichlet character $\chi_1$ mod $q$ being
uniquely determined by $\chi_1(g) = \e(1/(q-1))$,
where  $\e(x):=\exp(2\pi i x)$.  
Using $\chi_1$ we can 
represent the set of non-trivial characters mod $q$ as $\{\chi_1^j \colon j=1,2,\dotsc,q-2\}$.
Hence, if for every $k\in \{0,\dotsc,q-2\}$, we write $g^k\equiv a_k\in\{1,\dotsc,q-1\}$,
the innermost summation in \eqref{chi-Bernoulli-method-formula2}
is  of the type 
$\sum_{k=0}^{q-2}  
\e( j k /(q-1)) f(a_k/q)$,  $j\in\{1,\dotsc,q-2\}$
is odd  and $f$ is a suitable function. 
As a consequence, such quantities are the Discrete Fourier Transform (DFT)  
of the sequence $\{ f(a_k/q)\colon k=0,\dotsc,q-2\}$. 
This observation is due to Rader \cite{Rader1968} and it was used in 
\cite{FLM,Languasco2021a,Languasco2021b,LanguascoR2021}
to speed up the computation of similar quantities.

In our case we can also use the \emph{decimation in frequency} strategy: following the line 
of reasoning in  \cite[Section~8]{Languasco2023}, letting $\overline{q}=(q-1)/2$
for every $j=0,\dotsc, q-2$, $j=2t+\ell$, $\ell\in\{0,1\}$ and $t\in \Z$, we have that  
\begin{align}
\notag
\sum_{k=0}^{q-2} \e\Bigl(\frac{ j k}{q-1}\Bigr)  f \Bigl(\frac{a_k}{q}\Bigr)
&=
\sum_{k=0}^{\overline{q}-1}  
\,\e\Bigl(\frac{ t k}{\overline{q}}\Bigr)
\,\e\Bigl(\frac{\ell k}{q-1}\Bigr)
\Bigl(
f\Bigl(\frac{a_k}{q}\Bigr) 
+
(-1)^{\ell} 
f \Bigl(\frac{a_{k+\overline{q}}}{q}\Bigr)
\Bigr)
\\&
\label{DIF} 
=
\begin{cases}
\sum\limits_{k=0}^{\overline{q}-1} \,\e\bigl(\frac{t k}{\overline{q}}\bigr) b_k 
& \textrm{if} \ \ell =0;\\[2ex]
\sum\limits_{k=0}^{\overline{q}-1} \,\e \bigl(\frac{t k}{\overline{q}}\bigr)  c_k  
& \textrm{if} \ \ell =1,\\
\end{cases}
\end{align}
where $t=0,\dotsc, \overline{q}-1$,
\[
b_k :=
f\Bigl(\frac{a_k}{q}\Bigr) +  f \Bigl(\frac{a_{k+\overline{q}}}{q}\Bigr)   
\quad
\textrm{and}
\quad
c_k :=  
\e\Bigl(\frac{k}{q-1}\Bigr)   
\Bigl(  f\Bigl(\frac{a_k}{q}\Bigr) -  f \Bigl(\frac{a_{k+\overline{q}}}{q}\Bigr)  \Bigr).
\]
Since we just need the sum  over the odd Dirichlet characters for $f(x)=x$,
instead of computing an FFT  of length $q-1$
we can evaluate  an FFT of half a length, applied on the sequence
$c_k$ defined in \eqref{DIF}.
Clearly this leads to a gain in speed and in a reduction in memory usage.
In our case, using again $\langle g \rangle = (\Z/q\Z)^{*}$, $a_k \equiv g^k \bmod q$ and   $g^{\overline{q}} \equiv q-1 \bmod{q}$,
we can write that $ a_{k+\overline{q}}  \equiv  q-a_{k} \bmod{q}$;  hence
\[
a_k  -a_{k+\overline{q}}  
= 
a_k -(q-a_{k})  
=
2a_k  -q,
\]
so that we obtain $c_k= \e(k / (q-1))(2a_k/q -1)$ for every $k=0,\dotsc, 
\overline{q}-1$, $\overline{q}=(q-1)/2$.

\subsection{FFT accuracy estimate}
\label{FFT-accuracy}
In order to estimate the accuracy in performing the FFT procedure, we have to
recall first the definition of the Discrete Fourier Transform (DFT).
\begin{Defi}[The Discrete Fourier Transform $D$]
Let  $u_k\in \C$, $k=0, \dotsc, N-1$, be a sequence.  We define
the \emph{Discrete Fourier Transform} $D$ of $u_k$ as the following sequence
\[
\bigl( D(u_k)\bigr)_j := \sum_{k=0}^{N-1} u_k  \,\e\Bigl(-\frac{jk}{N}\Bigr) 
\]
where $j=0, \dotsc, N-1$ and  $\e(x)=\exp(2\pi i x)$.
The corresponding \emph{inverse Discrete Fourier Transform} 
$D^{-1}$ is defined as the sequence
\[
\bigl(D^{-1}(u_k) \bigr)_j
= 
\Bigl(
\frac{1}{N}\overline{D(\overline{u_k})}
\Bigr)_j
=
\Bigl(
\frac{1}{N}
\sum_{k=0}^{N-1} u_k  \,\e\Bigl(\frac{jk}{N}\Bigr)
\Bigr)_j,
\]
where $j=0, \dotsc, N-1$.
\end{Defi}
It is not hard to prove that  $D,D^{-1}$ are linear,
$D(D^{-1}(u_k)) = u_k$, $D^{-1}(D(u_k)) =u_k$,
and  $D/\sqrt{N},\sqrt{N}D^{-1}$ are $L^2$-isometries.

We also recall that the Fast Fourier Transform $F$ is an algorithm
that evaluates the Discrete Fourier Transform $D$.
We define analogously the Inverse Fast Fourier Transform $F^{-1}$.

According to Schatzman \cite[\S~3.4, p.~1159-1160]{Schatzman1996},  
the root mean square relative error $\mathcal E$ in the FFT satisfies 
\begin{equation}
\label{Delta-FFT}
\mathcal E=\frac{\Vert F(u_k) - D(u_k)\Vert_2}{\Vert D(u_k) \Vert_2}<\Delta,\text{~with~}\Delta:=0.6 \eps 
\Bigl(\frac{\log N}{\log 2}\Bigr)^{1/2},
\end{equation}
where $\eps$ is the machine epsilon and $N$ is the length of the transform. 
Moreover, the estimate in \eqref{Delta-FFT} holds for both $F^{-1}, D^{-1}$ too.

According to the IEEE 754-2008 specification, we can 
set $\eps=2^{-64}$ for the \emph{long double precision}  ($80$ bits)
of the C programming language. 
So for the largest case we are considering, $q=9854964401$, $N=(q-1)/2$, we get that $\Delta<1.85 \cdot 10^{-19}$. 
To evaluate the euclidean norm of the error we have then to multiply $\Delta$ and the
euclidean norm of 
$x_k:= 2 a_k/q-1$,
where  $a_k  = g^k \bmod q$, $\langle q \rangle$ 
$= (\Z/q\Z)^{*}$.
A straightforward computation gives 
\[
\Vert x_k \Vert_2 = \Bigl(\frac{(q-1)(q-2)}{6q}\Bigr)^{1/2}  
=40527.69505\dotsc
\]
Exploiting \eqref{Delta-FFT} and that  $D/\sqrt{N}$ is an $L^2$-isometry, we also obtain 
\begin{equation}
\label{RMS-FFT}
\Vert F(x_k) - D(x_k)\Vert_2
<
\Delta 
\Vert  
D(x_k)  
\Vert_{2}
=\Delta \sqrt{N}
\Vert  
x_k
\Vert_{2}.
\end{equation}
Recalling that $\Vert \cdot \Vert_{\infty} \le \Vert \cdot \Vert_{2}$ and using \eqref{RMS-FFT},
we can estimate that the maximal error in its FFT-computation
for this sequence is bounded by $ 7.48 \cdot 10^{-15}$ (long double precision case).  

We also estimated \emph{in practice} the accuracy in 
the actual computations  using the FFTW software library by evaluating at   
run-time the quantity
$
E_{j}(x_k) : = 
\Vert F^{-1}(F(x_k)) - x_k \Vert_{j},
$ 
$j\in \{2,\infty \}$.
Note that this quantity becomes zero if we replace $(F, F^{-1})$ by $(D,D^{-1})$.
Moreover, we also remark that 
from \eqref{RMS-FFT} we obtain
\begin{equation}
\label{RMS-FFT-2}
\Vert  
F(x_k)  
\Vert_{2}
\le
\Vert  
F(x_k)  - D(x_k)
\Vert_{2}   
+ 
\Vert  
D(x_k)  
\Vert_{2}
<
(1+\Delta) \Vert  
D(x_k)  
\Vert_{2}
=
(1+\Delta)  \sqrt{N}\Vert  
x_k
\Vert_{2}.
\end{equation}
We also have that
\begin{align}
\notag
E_{2}(x_k) 
&
=
\Vert F^{-1}(F(x_k)) - D^{-1}(D(x_k) )\Vert_{2}
\\
\notag
&
\le
\Vert
F^{-1}(F(x_k)) -  D^{-1}(F(x_k) ) 
\Vert_{2}
+  
\Vert 
D^{-1} [F(x_k)  - D(x_k) ]
\Vert_{2}
\\
\notag
&
<
\Delta
\Vert  D^{-1}(F(x_k) ) 
\Vert_{2}
+  
\frac{1}{\sqrt{N}}
\Vert 
F(x_k)  - D(x_k) 
\Vert_{2} 
\\
\label{back-forth-estim}
&
<
\frac{\Delta}{\sqrt{N}}
\Vert   F(x_k)  
\Vert_{2} 
+  
\Delta
\Vert  
x_k 
\Vert_{2}
<
\Delta(2+\Delta)
\Vert  
x_k 
\Vert_{2}
,
\end{align}
in which we used that $D^{-1}$ is linear 
and $ \sqrt{N}D^{-1}$ is an $L^2$-isometry, \eqref{Delta-FFT} for $F^{-1}$, 
and \eqref{RMS-FFT}-\eqref{RMS-FFT-2}.
Recalling that $\Vert \cdot \Vert_{\infty} \le \Vert \cdot \Vert_{2}\le  \sqrt{N} \Vert \cdot \Vert_{\infty}$, 
we also have $E_{\infty}(x_k) < \Delta(2+\Delta) \sqrt{N} \Vert x_k \Vert_{\infty}$.
For  $q=9854964401$, $N=(q-1)/2$ and $\eps = 2^{-64}$ in \eqref{Delta-FFT},
we get that $\Delta(2+\Delta)< 3.70 \cdot 10^{-19}$
and, using again the previous norm-computation, 
we also obtain that $E_{\infty}(x_k)< 1.50 \cdot 10^{-14}$.
Moreover, the actual computations using FFTW gave the following results:  
\begin{align*}
\frac{E_2(x_k)}{\Vert x_k \Vert_2} < 6.27 \cdot 10^{-19} , \quad
E_\infty(x_k) < 1.72 \cdot 10^{-18} , \quad
\end{align*}
in agreement with   \eqref{back-forth-estim}.

Summarising, we can conclude that at least ten decimals 
of our final results are correct.
If necessary, more accurate results can be obtained using the 
\emph{quadruple precision} ($128$ bits), which 
enables us to choose $\eps = 2^{-113}$ in \eqref{Delta-FFT}
and in the following argument, at the cost of a much slower practical execution. 

A similar, but shorter, analysis of the accuracy of the  
FFT procedure 
is performed in \cite[Section~5.7]{LanguascoR2021}.

\subsection{Comments on the plots and on the histograms} 
The actual values of $r(q)$, $R(q)$ and of all the other 
relevant quantities presented in the herewith included plots and histograms
were obtained for every odd prime $q$ up to $\bound$
using the FFTW \cite{FrigoJ2005} software library
set to work with the \textit{long double precision} (80 bits). 
Such results were then collected in some \emph{comma-separated values} (csv) files and then
all the plots and the histograms were
obtained running on such stored data some suitable designed scripts written
using \pythonversion\ and making use of the packages \pandasversion\ and \matplotlibversion.
The normal function used in the histograms is defined as 
$$\mathcal{N}(x,\mu,\sigma) := \frac{1}{\sigma \sqrt{2\pi}} \exp\Bigl(- \frac{(x-\mu)^2}{2\sigma^2}\Bigr),$$
where $\mu$ and $\sigma$ 
denote the mean, respectively  the
standard deviation, of the plotted data.

Figures \ref{fig1} shows some values for $R(q)$.
The first important remark is that Figure \ref{fig2}
shows that  $r(q)$ has a symmetrical distribution having average equal to $0$;
this symmetry is not shown by the known theoretical results.
The frequent values of $r(q)$ are rational numbers $r$ for 
which the smallest set $\mathcal A$ with $\mu(\mathcal A)=2r$ has very few elements. 
This explains the ``concentration'' of the computed data
around the values $\pm1/4$
(with $\mu(\{2\})=\frac12$
and primes $q$ such that
$2q\pm1$ contributing, see Figure \ref{fig3}).
Moreover, the peak
around the values $\pm1/8$ depends on 
the contribution of the primes $q$ such that
$4q\pm1$ are primes.
In both cases, considering only the contribution of the primes 
$q$ such that
$2q\pm 1$ and  $4q\pm 1$ are composite produces distributions
very similar to the normal one, see Figures \ref{fig4}. 

The numerical results mentioned, the plots and the histograms,
and the programs used to obtain $R(q)$ and $r(q)$, are available
at \url{https://www.math.unipd.it/~languasc/rq-comput-reprise.html}

\medskip
\noindent \textbf{Acknowledgment}. 
Part of the work was done during the postdocs and multiple visits of the
first, fourth and fifth author
at the Max Planck Institute for Mathematics (MPIM) under the
mentorship of third author. They thank MPIM for the invitations, 
the hospitality of the staff and the excellent working conditions.
The fourth author is supported by the Austrian Science Fund (FWF): P 35863-N.
The computational work was carried out on machines
of the cluster located at the Dipartimento di Matematica ``Tullio Levi-Civita'' of 
the University of Padova, see \url{https://hpc.math.unipd.it}. 
The authors are grateful for having had such computing facilities 
at their disposal. 
The fifth author wishes to thank the INI and LMS for the financial support.

\begin{table}[htp]
\scalebox{0.625}{
\begin{tabular}{|c|c|}
\hline
$q$  & $R(q)$\\ \hline
$3$ & $0.6045997880780726168646927525473$ \\ \hline
$5$ & $0.7895683520871486895067592799900$\\ \hline
$7$ & $0.9566751857508418754795073381317$\\ \hline
$11$ & $1.1091619128700057589698217531662$\\ \hline
$13$ & $1.0771490562098575674859781589187$\\ \hline
$17$ & $0.8553903456876526811590587393660$\\ \hline
$19$ & $0.7070400490038472907067462197858$\\ \hline
$23$ & $1.2730306993968550223440516296068$\\ \hline
$29$ & $1.1950722585472314170213869230139$\\ \hline
$31$ & $0.8898896210785440789198518157132$\\ \hline
$37$ & $0.8961735424518262426393010568398$\\ \hline
$41$ & $1.0109514928155133737670365161798$\\ \hline
$43$ & $1.0003280708398792157908433519393$\\ \hline
$47$ & $0.9951041947584376332046179459764$\\ \hline
$53$ & $1.0023154955608046980883540349743$\\ \hline
$59$ & $1.0311199595775858834174986891680$\\ \hline
$61$ & $0.9154168975763615203860784058478$\\ \hline
$67$ & $1.0323019630420196815155397633286$\\ \hline
$71$ & $0.9465247471036236809290054627120$\\ \hline
$73$ & $1.2821779323076053838224676118514$\\ \hline
$79$ & $0.8457945961200297550455294076382$\\ \hline
$83$ & $1.2232692654844146161950762139016$\\ \hline
$89$ & $1.2863214746192234623445369458997$\\ \hline
$97$ & $0.9046761428702376506678185793342$\\ \hline
$101$ & $1.1104995875358644805192388808229$\\ \hline
$103$ & $1.0556519883371874318616371348168$\\ \hline
$107$ & $0.9926076779267250130951961237566$\\ \hline
$109$ & $0.9155428388523018685066750024637$\\ \hline
$113$ & $1.1618557363506180805776111458998$\\ \hline
$127$ & $1.0626983549971763540798019088845$\\ \hline
$131$ & $1.2789769938976286727059298824683$\\ \hline
$137$ & $1.0018885365042079285157114283333$\\ \hline
$139$ & $0.8716611518739232788670854213024$\\ \hline
$149$ & $1.0488652764269119456479100644937$\\ \hline
$151$ & $1.0961352605053081203560323292152$\\ \hline
$157$ & $0.7430450532910889660052300286210$\\ \hline
$163$ & $0.9516739236944299288308183830698$\\ \hline
$167$ & $0.8540489171409883518683860745104$\\ \hline
$173$ & $1.2575031110060486325647665223234$\\ \hline
$179$ & $1.3189895521869900854067212054754$\\ \hline
$181$ & $1.0164672530790178324085643879748$\\ \hline
$191$ & $1.2985095534724676367615527171504$\\ \hline
$193$ & $1.1738495661428052368362517610841$\\ \hline
$197$ & $0.8714268580587022585427508674145$\\ \hline
$199$ & $0.7977576598180326170333641097002$\\ \hline
$211$ & $0.7096581038457700773915382688127$\\ \hline
$223$ & $0.9001673677400910738942007486095$\\ \hline
$227$ & $0.7629883976313712260376287117080$\\ \hline
$229$ & $0.7241457414201049462008640419682$\\ \hline
$233$ & $1.4310221673105806346958377026375$\\ \hline
$239$ & $1.1852025922101838102852657887109$\\ \hline
$241$ & $1.1190819269965132548112076907794$\\ \hline
$251$ & $1.1804169442539285917038758350886$\\ \hline
$257$ & $0.9055962573549657664091346453876$\\ \hline
$263$ & $0.9371707816685296065406493231972$\\ \hline
$269$ & $1.0105242994134286604110488301351$\\ 
\hline
\end{tabular}
}
\scalebox{0.625}{
\begin{tabular}{|c|c|}
\hline
$q$  & $R(q)$\\ \hline
$271$ & $0.8412088090144110303458717890667$\\ \hline
$277$ & $1.2228716770080365999632534704580$\\ \hline
$281$ & $1.0907231267144641150745775682060$\\ \hline
$283$ & $0.9873004592498935117673519297087$\\ \hline
$293$ & $1.2884302359523728319105679845501$\\ \hline
$307$ & $0.9135872522019948222051491689937$\\ \hline
$311$ & $1.1458937454264730221344414268718$\\ \hline
$313$ & $0.9389331767581916618067398442288$\\ \hline
$317$ & $0.8067182318898481284945719857774$\\ \hline
$331$ & $0.8135627495605184590233164933650$\\ \hline
$337$ & $0.8611151152192259126883225579098$\\ \hline
$347$ & $1.0851794175810526744648331305833$\\ \hline
$349$ & $0.9839573134487701044559123913262$\\ \hline
$353$ & $0.8860350566174460450308781577592$\\ \hline
$359$ & $1.1600264444670825456691643273527$\\ \hline
$367$ & $0.9086410187793691206326531982541$\\ \hline
$373$ & $1.0750761442013325764626703553466$\\ \hline
$379$ & $0.7214461864713844469442699568877$\\ \hline
$383$ & $0.8324380926742871396047076038085$\\ \hline
$389$ & $0.8499778289685450397162756349668$\\ \hline
$397$ & $0.9975778112015857909425324661679$\\ \hline
$401$ & $1.1399832831644707063138427893128$\\ \hline
$409$ & $1.1991980974390954074874424479768$\\ \hline
$419$ & $1.1897445888237693592676697100177$\\ \hline
$421$ & $0.8645796653071174117734286953546$\\ \hline
$431$ & $1.1375426110359346246171708562349$\\ \hline
$433$ & $1.0717613518204177138545059520477$\\ \hline
$439$ & $0.6848413406172976205500589562641$\\ \hline
$443$ & $1.4108998843039798698090656834498$\\ \hline
$449$ & $0.9053964365861442489589154746074$\\ \hline
$457$ & $0.8373463419058562177863679134357$\\ \hline
$461$ & $1.0311955737739740364528472490666$\\ \hline
$463$ & $0.9613462511195984177868663523170$\\ \hline
$467$ & $0.8974045485919283687065708373771$\\ \hline
$479$ & $1.1050671578064206970591097893948$\\ \hline
$487$ & $1.1304102278265606313945369715559$\\ \hline
$491$ & $1.2722146569130496835275435498496$\\ \hline
$499$ & $0.8297902495946506366988138268051$\\ \hline
$503$ & $1.0995617471957832909336221046588$\\ \hline
$509$ & $1.3969208271961266132041741065091$\\ \hline
$521$ & $0.7448857918191827286091015924803$\\ \hline
$523$ & $0.9951484787399289420380269322079$\\ \hline
$541$ & $0.9447265578295298152134577952949$\\ \hline
$547$ & $0.7386850547619545899616661320191$\\ \hline
$557$ & $1.0180061813097044024347867514026$\\ \hline
$563$ & $0.9232212509133752364416200184615$\\ \hline
$569$ & $0.8664438451435738527270516828484$\\ \hline
$571$ & $0.9966248063685197276230915134980$\\ \hline
$577$ & $0.9137029380401851023927738920458$\\ \hline
$587$ & $0.8125245985067212166037417395454$\\ \hline
$593$ & $1.0773461748966493078075918844172$\\ \hline
$599$ & $0.9640877383472306977957126847174$\\ \hline
$601$ & $0.9282733975182409725085430055023$\\ \hline
$607$ & $0.8363731270525144324766779910174$\\ \hline
$613$ & $0.8770365930347214891035502029408$\\ \hline
$617$ & $0.8424608454194671614144537884806$\\ 
\hline
\end{tabular}
}
\scalebox{0.625}{
\begin{tabular}{|c|c|}
\hline
$q$  & $R(q)$\\ \hline
$619$ & $0.8046391863654823181809704923832$\\ \hline
$631$ & $1.1396469807244276647958444773064$\\ \hline
$641$ & $1.3429915643232847544526367324545$\\ \hline
$643$ & $1.0183620561136068530441755349378$\\ \hline
$647$ & $0.9023366731711887559549077220931$\\ \hline
$653$ & $1.2708772780577246646879609833837$\\ \hline
$659$ & $1.3910631789822655014399826852608$\\ \hline
$661$ & $0.8354443097523214656936838597289$\\ \hline
$673$ & $1.0366020698239863718118735321088$\\ \hline
$677$ & $0.9242401331249736440179204466235$\\ \hline
$683$ & $1.1352828140240947699825469113423$\\ \hline
$691$ & $0.7692142795745405069640641103691$\\ \hline
$701$ & $0.9208988286796986104162625438273$\\ \hline
$709$ & $1.0564893491780186160617480034127$\\ \hline
$719$ & $1.2030632585533392768111724372893$\\ \hline
$727$ & $0.9985692142278032863134063960758$\\ \hline
$733$ & $0.9801491017726198673607802262180$\\ \hline
$739$ & $1.1026354682405308663067124546439$\\ \hline
$743$ & $1.0349549409620577590409117683058$\\ \hline
$751$ & $1.0185620058358507387809584897536$\\ \hline
$757$ & $0.9670687611870859854554145544772$\\ \hline
$761$ & $1.4695828581314155249132265698413$\\ \hline
$769$ & $0.8989223036739211131497271647474$\\ \hline
$773$ & $1.0681094719703144713033330503520$\\ \hline
$787$ & $0.9717823284398633668645155647478$\\ \hline
$797$ & $1.0307513038736094294364198652655$\\ \hline
$809$ & $1.3197044140601871225194956764480$\\ \hline
$811$ & $0.8028381726481542070785681890576$\\ \hline
$821$ & $1.0652843703654964331935281465584$\\ \hline
$823$ & $0.9676931847618204865646570591849$\\ \hline
$827$ & $0.8655599367575844205769196995319$\\ \hline
$829$ & $0.8225003354161554974840091964490$\\ \hline
$839$ & $0.9187109054076576161004431766675$\\ \hline
$853$ & $1.0822358288025334754800428361362$\\ \hline
$857$ & $1.0507531149069469457639638202726$\\ \hline
$859$ & $0.8808009418056817847639767572161$\\ \hline
$863$ & $1.0569423120644476418024040128851$\\ \hline
$877$ & $0.7228939852270574121828463785409$\\ \hline
$881$ & $1.0973899419907535018443533637062$\\ \hline
$883$ & $1.1331822763939321498203901268480$\\ \hline
$887$ & $0.9691797419679082310841771993673$\\ \hline
$907$ & $0.9026255886631148047805162360749$\\ \hline
$911$ & $1.0779855753630487309935104370070$\\ \hline
$919$ & $1.0400334655419995090131730345736$\\ \hline
$929$ & $1.0441490445298916774481320172549$\\ \hline
$937$ & $0.9001793485775001978413226252376$\\ \hline
$941$ & $1.0940086717975223552339721484683$\\ \hline
$947$ & $1.2258744827051074302609149043348$\\ \hline
$953$ & $1.1608317303128388568222684560130$\\ \hline
$967$ & $0.7286000440466886148143682504791$\\ \hline
$971$ & $1.0793911591644004625871038160386$\\ \hline
$977$ & $0.8389088588037128235412547247521$\\ \hline
$983$ & $0.7886767720297385404724656676372$\\ \hline
$991$ & $0.9094393615350512976006963975090$\\ \hline
$997$ & $0.8557575449135065446654521786495$\\ \hline
\phantom{} & \phantom{}  \\ 
\hline
\end{tabular}
}
\caption{\label{table1}
Values of $R(q)$ (truncated) for every odd prime up to $1000$.
}
\end{table}  

\clearpage
\begin{table}[H]
\hskip-1cm
\scalebox{0.65}{ 
\begin{minipage}{0.55\textwidth} 
\begin{tabular}{|c|c|}
\hline   
$q$  & $R(q)$ \\  \hline 
$1451$ & $1.489316072080934425611321346752\dotsc$ \\ \hline 
$2741$ & $1.498121015176665823721124535220\dotsc$ \\ \hline 
$3331$ & $0.642429297634719506688741152270\dotsc$ \\ \hline 
$4349$ & $1.518570512426339397454202981116\dotsc$ \\ \hline  
$4391$ & $1.507776410131052825600361832032\dotsc$ \\ \hline
$5231$ & $1.556562247546690554629305894110\dotsc$ \\ \hline
$6101$ & $1.511405291132409881116244836469\dotsc$ \\ \hline 
$6379$ & $0.673523026278795404982148735902\dotsc$ \\ \hline
$7219$ & $0.658084090096317378291742795450\dotsc$ \\ \hline 
$8209$ & $0.672045039003857595919734222943\dotsc$ \\ \hline 
$9049$ & $0.667614244171116232015569216575\dotsc$ \\ \hline 
$9689$ & $1.524371504087494924535704793958\dotsc$ \\  
\hline
\end{tabular}
\end{minipage}
\begin{minipage}{0.55\textwidth} 
\begin{tabular}{|r|l|}  
\hline 
$q$ \phantom{012}  &\phantom{0123456789012} $R(q)$  \\
\hline
$4391$ & $1.507776410131052825600361832032\dotsc$\\\hline
$5231$ & $1.556562247546690554629305894110\dotsc$ \\ \hline
$42 611$&  $1.619906571157532399867361172777\dotsc$ \\ \hline
$198 221$  &  $1.623477270751197661500864242418\dotsc$ \\ \hline
$305 741$  &   $1.661436485908786948688528096415\dotsc$ \\ \hline
\hline
$6 766 811$ &$1.7093790418\dotsc$ \\ \hline
$1 326 662 801$ & $1.7097585606\dotsc$ \\ \hline
$1 979 990 861$ & $1.7207910074\dotsc$ \\ \hline
$4 735 703 723$ & $1.7216545866\dotsc$ \\ \hline
$9 697 282 541$ & $1.7247411203\dotsc$ \\
\hline
\end{tabular} 
\vspace{0.8cm}
\end{minipage}
}
\caption{\label{table2}
On the left: few other values of $R(q)$. On the right: maximal champions for $R(q)$. 
The values for $q \le 305 741$  were obtained using PARI/GP with a trivial summation
over $a$ and an accuracy of $100$ decimals; the others using the FFTW software library ($128$ bits accuracy). 
} 
\end{table} 

\begin{table}[H]
\begin{center}
\scalebox{0.75}{
\begin{tabular}{|r|l|}  
\hline 
$q$ \hskip0.3truecm\mbox{}  & \hskip0.75truecm\mbox{} $R(q)$   \\
\hline
$37 189$  &  $0.625231255787654795233417601859\dotsc$ \\ \hline
$149 119$ &  $0.624149715978401425409347395847\dotsc$ \\ \hline
$401 179$ &  $0.621507092276527124572758370995\dotsc$ \\ \hline
\hline
$2 083 117$   & $0.6142798512\dotsc$ \\ \hline
$5 589 169$   & $0.5869729849\dotsc$ \\ \hline
$102 598 099$ & $0.5861372431\dotsc$ \\ \hline
$116 827 429$ & $0.5756742526\dotsc$ \\ 
\hline
\end{tabular} 
}
\caption{\label{table3}
Minimal champions for $R(q)$.
The values for $q \le 401 179$  were obtained using PARI/GP with a trivial summation
over $a$  and an accuracy of $100$ decimals; the others using the FFTW software library  ($128$ bits accuracy). 
}
\end{center}
\end{table}

\begin{figure}[H]
\includegraphics[scale=0.5,angle=0]{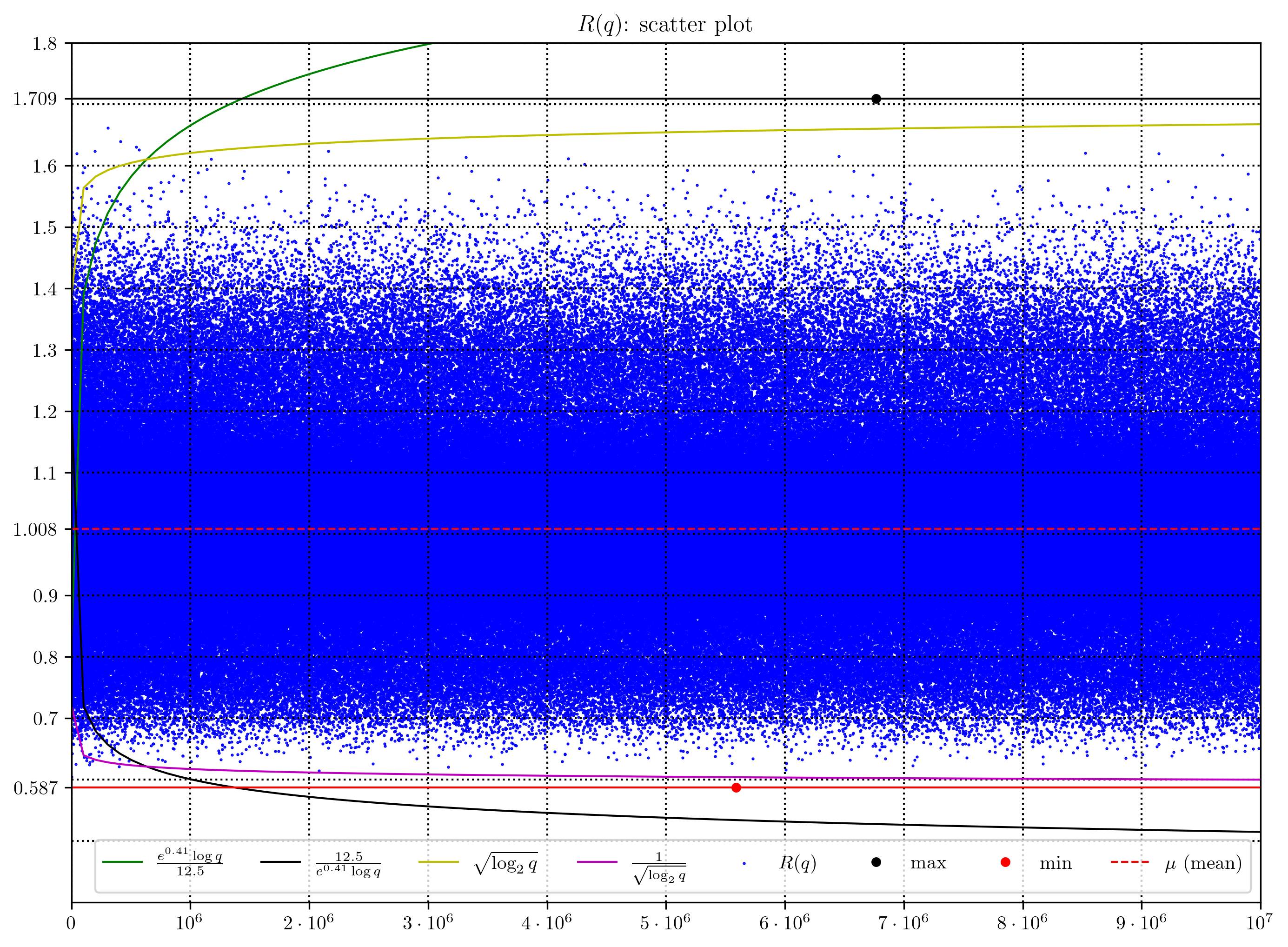} 
\caption{{\small The values of $R(q)$, $q$ prime, $3\le q\le  \bound$.
The red dashed line represents the mean value.
}}
\label{fig1}
\end{figure}

\begin{figure}[H]
\hskip-1.25cm
\scalebox{0.85}{
\begin{minipage}{0.48\textwidth} 
\includegraphics[scale=0.35,angle=0]{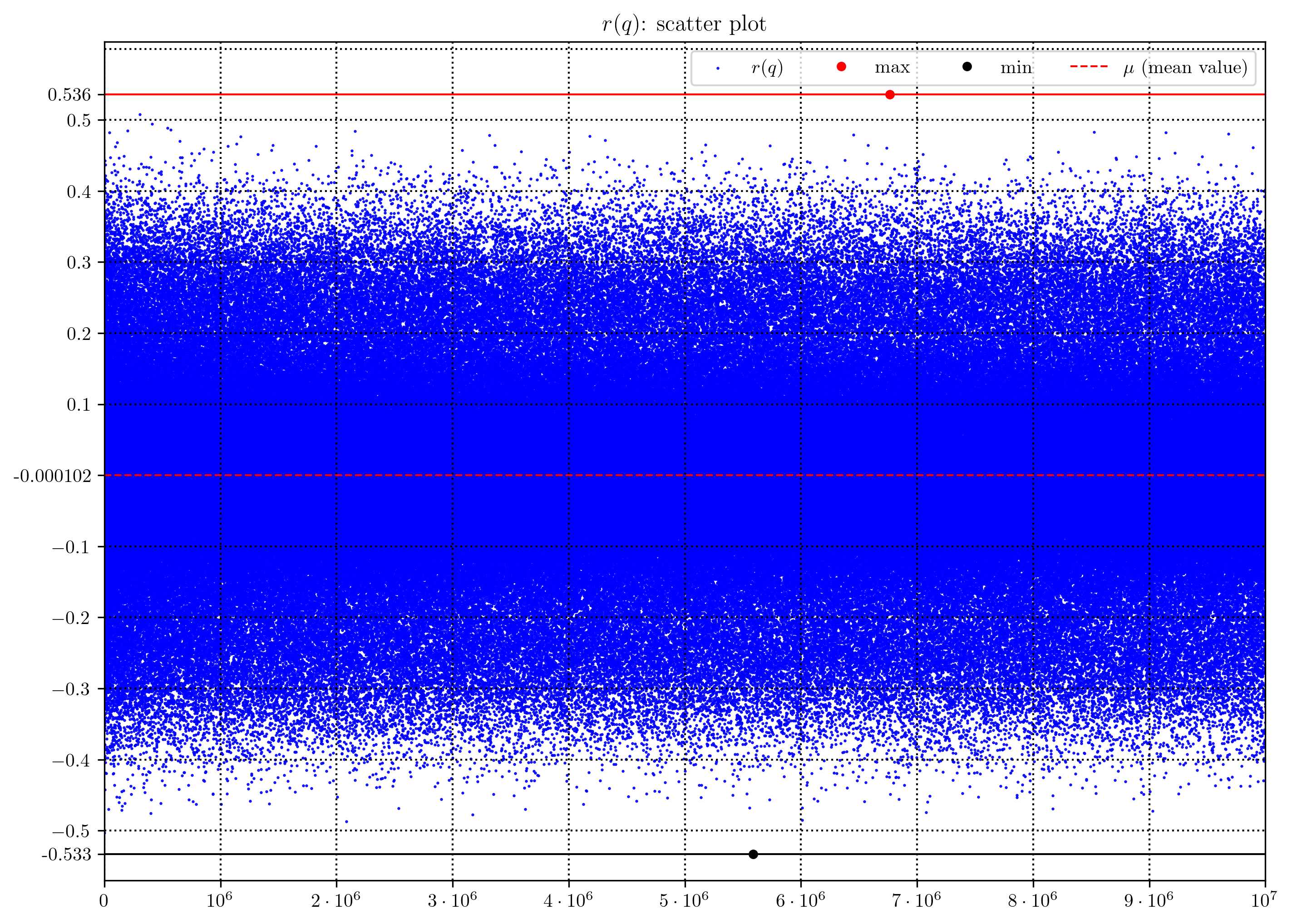}  
\end{minipage} 
\hskip1.25cm
\begin{minipage}{0.5\textwidth} 
\includegraphics[scale=0.56,angle=0]{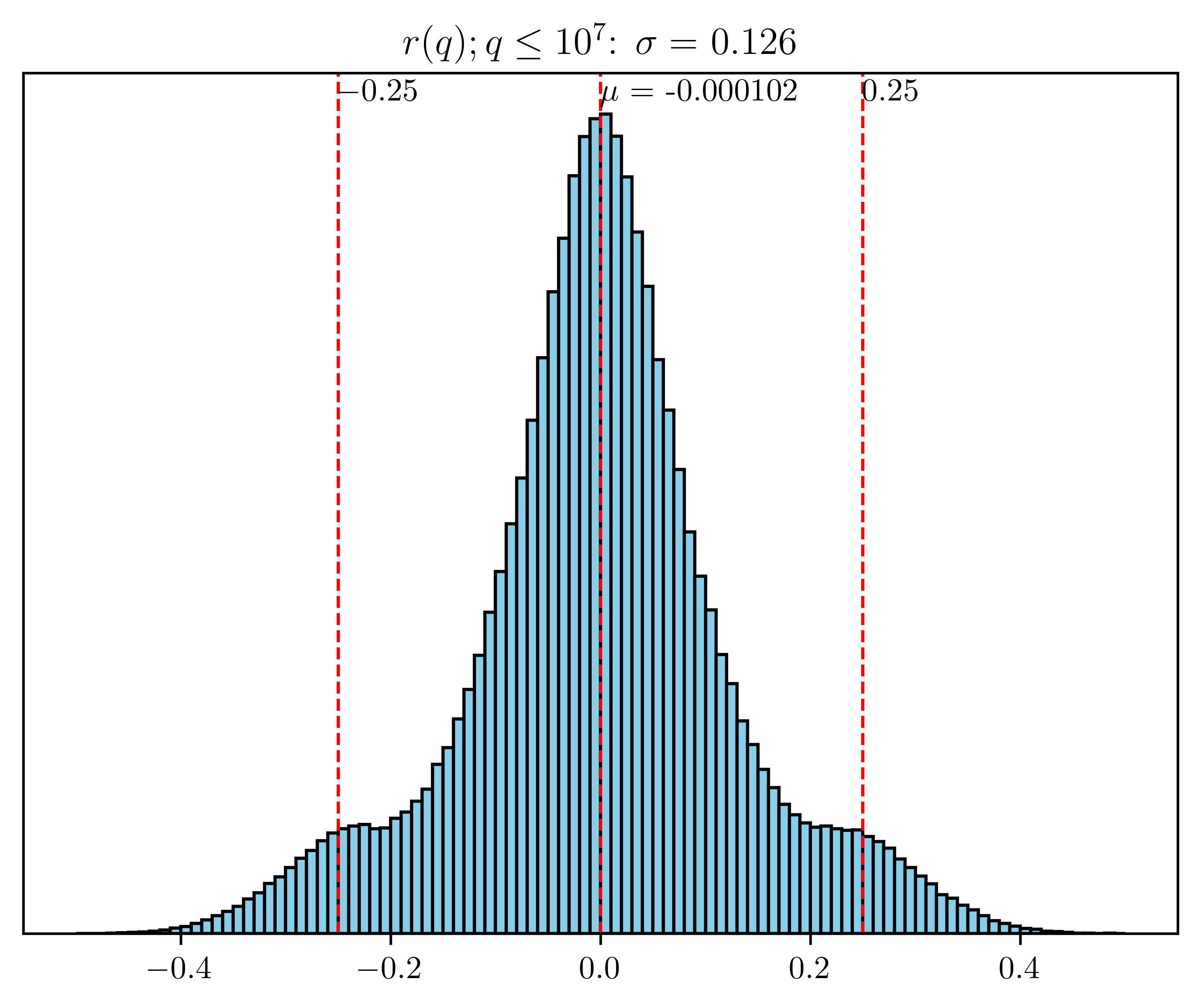} 
\end{minipage} 
}
\caption{{\small On the left: the values of $r(q)$;
}}
\label{fig2}
\end{figure}

\begin{figure}[H]
\scalebox{0.85}{
\begin{minipage}{0.48\textwidth} 
\includegraphics[scale=0.56,angle=0]{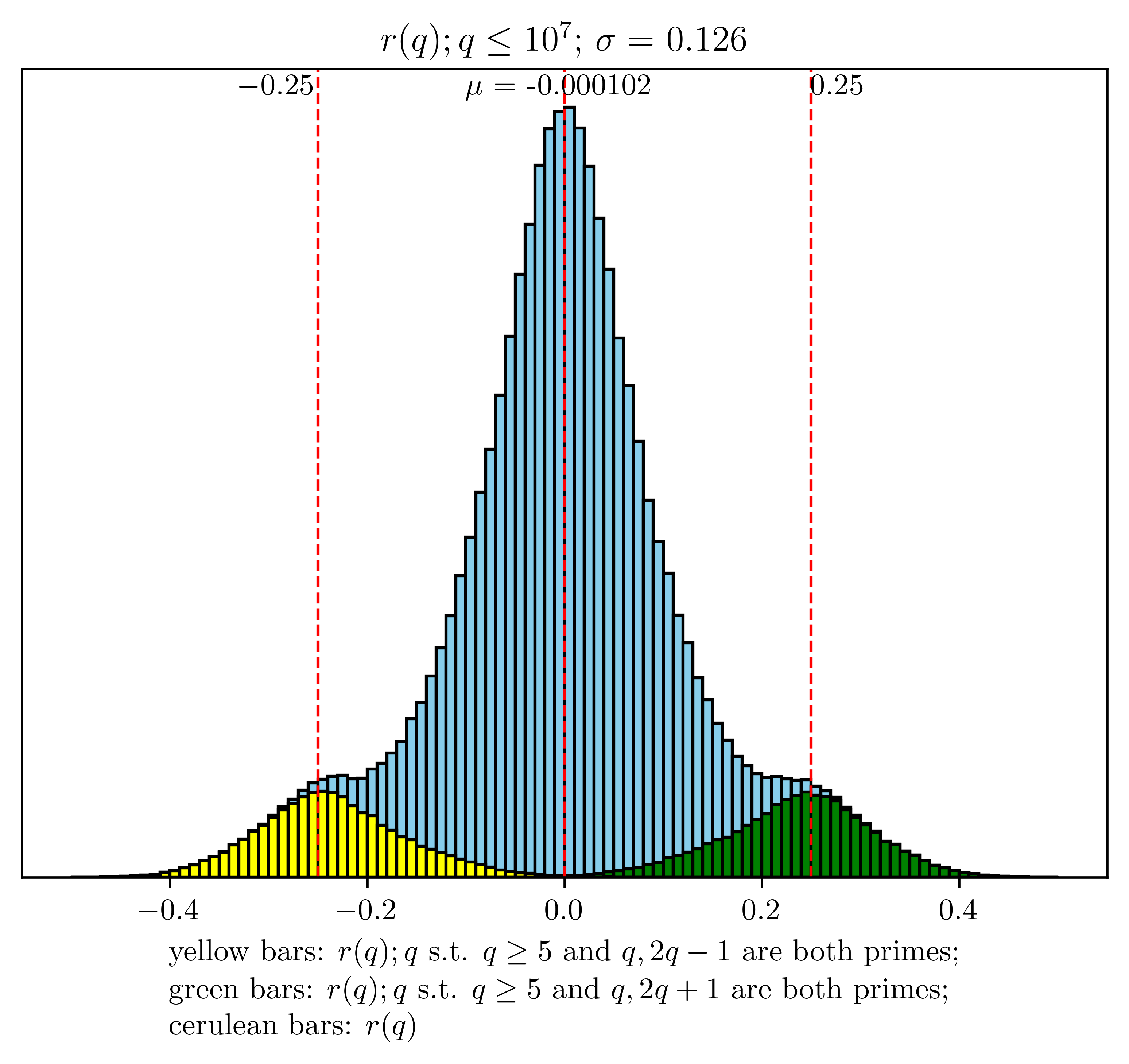}  
\end{minipage} 
\begin{minipage}{0.48\textwidth} 
\includegraphics[scale=0.56,angle=0]{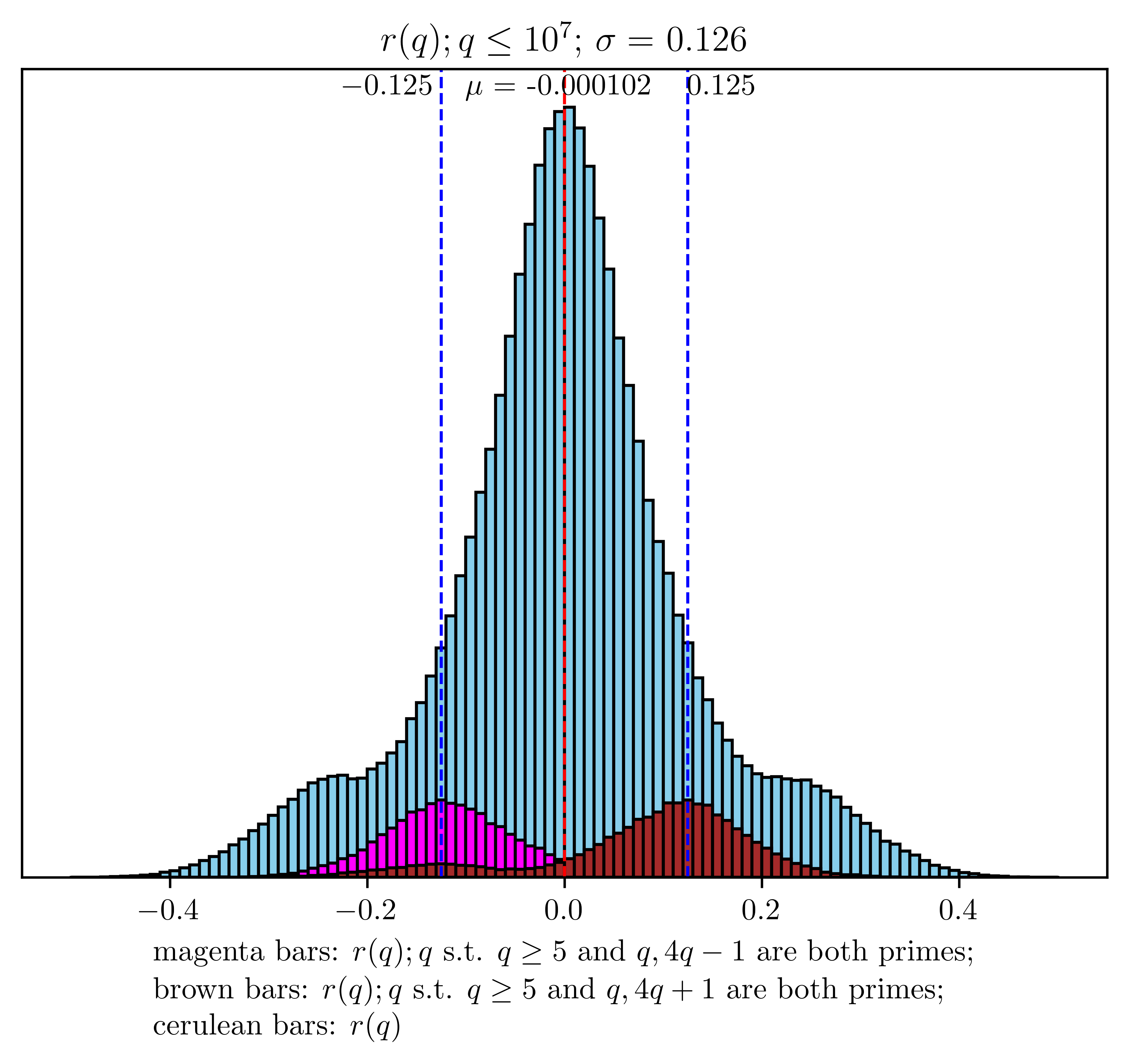}  
\end{minipage} 
}
\caption{On the left: the same histograms of Figure \ref{fig2} but the contributions of the primes $q$ such that 
$2q+1$ is prime or $2q-1$ is prime (the ``spikes'') are superimposed.
On the right: the contributions of the primes $q$ such that 
$4q+1$ is prime or $4q-1$ is prime (the ``spikes'') are superimposed.
}
\label{fig3}
\end{figure}  

 \begin{figure}[H]
 \scalebox{0.85}{
\begin{minipage}{0.48\textwidth} 
\includegraphics[scale=0.56,angle=0]{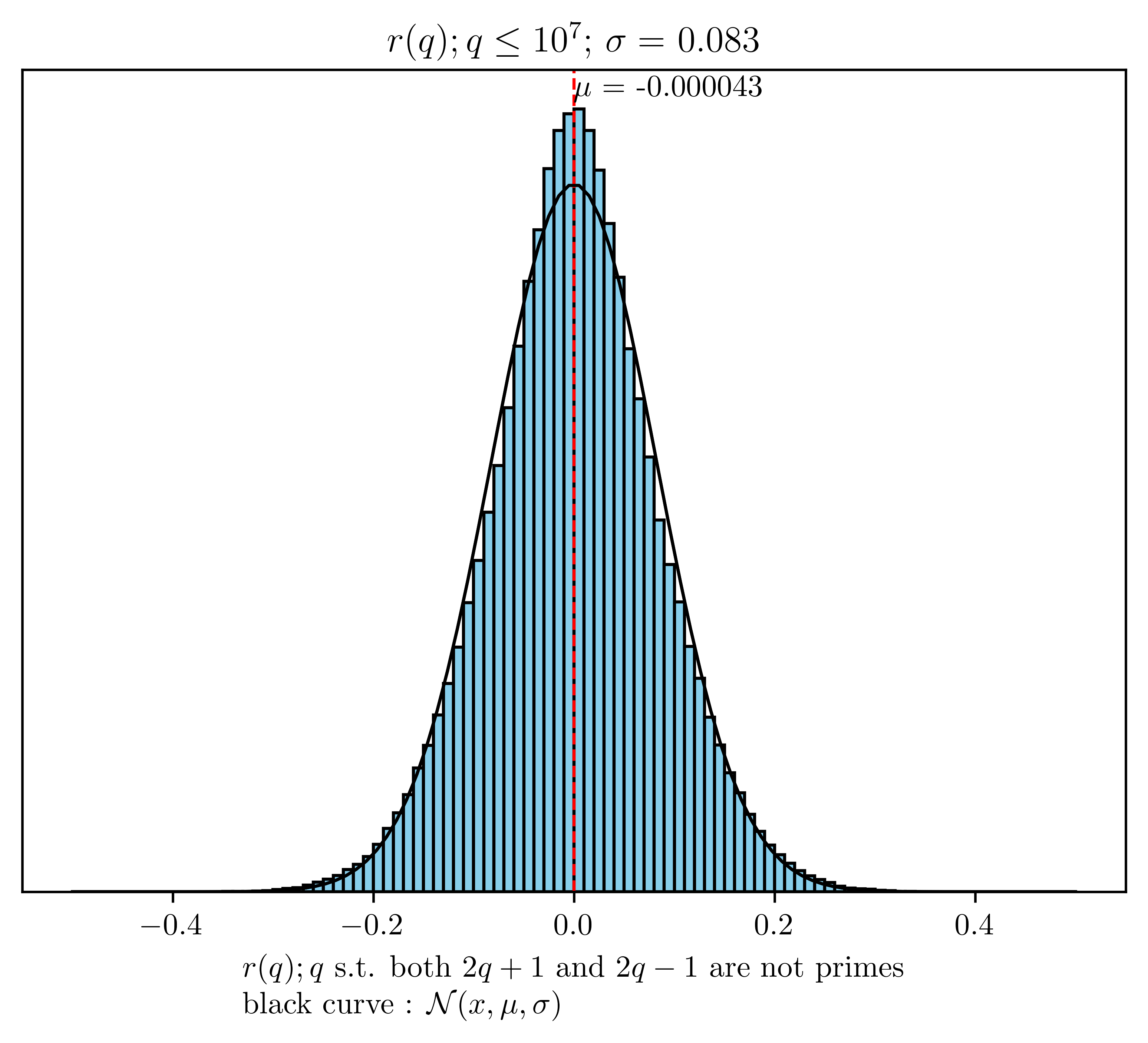}
\end{minipage} 
\begin{minipage}{0.48\textwidth} 
\includegraphics[scale=0.56,angle=0]{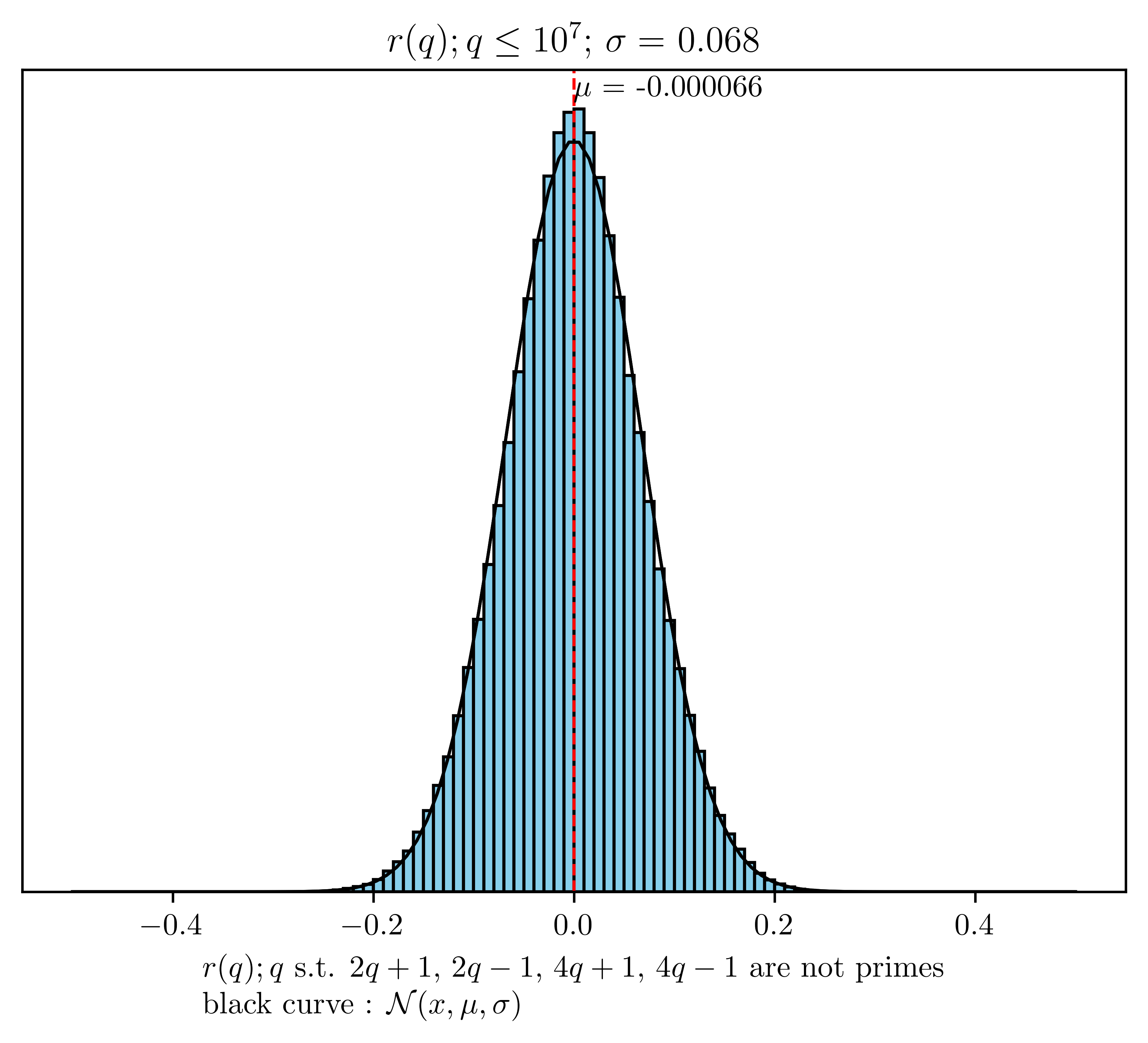}
\end{minipage} 
}
\caption{On the left: the histogram for $r(q)$, $q$ prime, $5\le q\le  \bound$, such that
$2q\pm1$ are composite; on the right: the same with both $2q\pm1$ and $4q\pm1$ that are composite numbers.
The red dashed lines represent the mean values.}
\label{fig4}
\end{figure}


\bigskip
\bigskip\noindent Neelam Kandhil  \par\noindent
{\footnotesize Max-Planck-Institut f\"ur Mathematik,
Vivatsgasse 7, D-53111 Bonn, Germany.\hfil\break
e-mail: {\tt kandhil@mpim-bonn.mpg.de}}

\medskip\noindent Alessandro Languasco \par\noindent
{\footnotesize Universit\`a di Padova, Dipartimento di Matematica ``Tullio Levi-Civita'', via Trieste 63, 35121 Padova, Italy.\hfil\break
e-mail: {\tt alessandro.languasco@unipd.it}; {\tt alessandro.languasco@gmail.com}}

\medskip\noindent Pieter Moree  \par\noindent
{\footnotesize Max-Planck-Institut f\"ur Mathematik,
Vivatsgasse 7, D-53111 Bonn, Germany.\hfil\break
e-mail: {\tt moree@mpim-bonn.mpg.de}}

\medskip\noindent Sumaia Saad Eddin \par\noindent
{\footnotesize Johann Radon Institute for Computational and Applied Mathematics,\\
Austrian Academy of Sciences,
Altenbergerstrasse 69, A-4040 Linz, Austria.\hfil\break
e-mail: {\tt sumaia.saad-eddin@ricam.oeaw.ac.at}}

\medskip\noindent Alisa Sedunova \par\noindent
{\footnotesize 
Mathematics Institute, Zeeman Building. University of Warwick, Coventry CV4 7AL \hfil\break
e-mail: {\tt alisa.sedunova@gmail.com}}
\end{document}